\documentclass{article}
\usepackage{amssymb}
\usepackage{amsmath}
\usepackage{amsthm}
\usepackage{mathrsfs}
\usepackage{lineno}
\usepackage{pdflscape}
\usepackage{booktabs}
\usepackage{adjustbox}
\usepackage{tabularx}
\usepackage{array}
\usepackage{makecell}

\newcommand{\EC}{\mathbb{C}}
\newcommand{\MK}{\mathcal{K}}
\newcommand{\MAq}{\mathcal{A}^q}
\newcommand{\MA}{\mathcal{A}}
\newcommand{\MBq}{\mathcal{B}^q}
\newcommand{\MB}{\mathcal{B}}

\newcommand{\MMq}{\mathcal{M}^q}
\newcommand{\MKq}{\mathcal{K}^q}
\newcommand{\sQ}{\mathscr{Q}}
\newcommand{\sQe}{\mathscr{Q}^\epsilon}
\newcommand{\sRd}{\mathscr{R}^\delta}
\newcommand{\ijbar}{{i\overline{j}}}
\newcommand{\jbar}{{\overline{j}}}
\newcommand{\mineig}{\text{Minimal Eigenvalue}}
\newcommand{\smineig}{\text{Second Minimal Eigenvalue}}
\newcommand{\diag}{\text{diag}}
\newcommand{\lamb}{{\lambda_1}}
\newcommand{\samb}{{\lambda_2}}
\newcommand{\Ov}{	\Omega^\varepsilon}
\newtheorem{proposition}{Proposition}
\newcommand{\sigmae}{\sigma^\epsilon}
\newcommand{\pqbar}{{p\overline{q}}}
\newcommand{\Lsigma}{L^{\pqbar}\sigma_{\pqbar}}
\newcommand{\uonebar}{{u_{\overline1}}}

\newcommand{\qbar}{{\overline q}}
\newcommand{\rqbar}{\overline{r^q}}
\newcommand{\onebar}{{\overline{1}}}

\newcommand{\sbar}{{\overline{s}}}
\newcommand{\trace}{{\text{tr}}}

\newcommand{\MM}{\mathcal{M}}
\newtheorem{problem}{Problem}
\newcommand{\Ree}{\text{Re}}
\newcommand{\MAbarinverse}{{\overline{\MA}}^{-1}}
\newcommand{\MAbar}{{\overline{\MA}}}
\newcommand{\MBbar}{{\overline{\MB}}}
\newcommand{\deltaa}{{\delta_1}}
\newcommand{\deltab}{{\delta_2}}
\newcommand{\ba}{{\bf{a}}}
\newcommand{\bK}{{\bf K}}

\numberwithin{equation}{section}
\newtheorem{theorem}{Theorem}[section]
\newtheorem{lemma}[theorem]{Lemma}

\newtheorem{remark}{Remark}

\newcommand{\e}{\mathrm{e}}
\newcommand{\p}{\partial_{p}}
\newcommand{\q}{\partial_{\bar{q}}}
\newcommand{\pq}{\partial_{p\bar{q}}}
\newcommand{\pz}{\partial_{z}}
\newcommand{\pbz}{\partial_{\bar{z}}}
\newcommand{\zz}{\partial_{z\bar{z}}}
\newcommand{\ma}{\mathcal{A}}
\newcommand{\mb}{\mathcal{B}}
\newcommand{\mm}{\mathcal{M}}
\title{{Power Convexity of Solutions to the Complex Monge-Amp\`{e}re Equation $\det(u_{\ijbar})=1$ in Complex Dimension Two}}
	\author{Hongyu Chen, Jingchen Hu, Li Sheng} 
\begin{document}

		\maketitle 
		\begin{abstract}
			In this paper, we establish the power convexity of solutions to the complex Monge-Amp\`{e}re equation $\det(u_{\ijbar})=1$ in a convex set in the complex 2 dimensional Euclidean space.  Even the convexity result is only for 2 dimension, our approach is applicable to all dimensions.
		\end{abstract}

	\section{Introduction}\label{sec1}
	Convexity properties of solutions are a classical and fundamental theme in elliptic partial differential equations. It is interesting to investigate if the solution $u$ of a partial differential equation is convex, and, if not, can you find a suitable monotone function $g(t)$ so that $g(u)$ is  convex.

	This line of inquiry was profoundly advanced by the pivotal work of \cite{Caffarelli-Friedman1985}, who established the strict power convexity of solutions for a class of semilinear elliptic equations in real dimension two
	\begin{equation}
		\begin{cases}
			\Delta u=f(u) & \text{ in }\Omega, \\
			u=M                  & \text{ on }\partial\Omega,
		\end{cases}
	\end{equation}
	where $\Omega$ is a bounded  convex domain in $\mathbb{R}^2$ and $M$ is a real constant. 
	Since then, this technique has been extensively studied and generalized by many authors, such as \cite{maxu}, \cite{Bian-Guan2009}, \cite{Bian-Guan-Ma-Xu2011}, \cite{Guan-Xu2013}, \cite{SzekelyhidiWeinkove2016}, and \cite{Chen-Jia-Xiong2025}.
	
	Also, very recently, works of Chen-Ma-Li-Salani established the power convexity of solutions to real $\sigma_2$ Hessian equation \cite{RealSigma} and the logarithmic convexity of solutions to complex $\sigma_2$ Hessian eigenvalue problem \cite{ComplexSigmaEigen}.

	In this paper, we investigate the power convexity of solutions to the Dirichlet problem for the complex Monge-Amp\`{e}re equation. Specifically, let $\Omega$ be a strictly convex bounded domain in $\mathbb{C}^n$, and let 
	$ u : \overline{\Omega} \rightarrow \mathbb{R} $ be the solution of the following Dirichelt problem:
	\begin{problem}\label{prob:Dirichlet}
	\begin{equation}\label{main equation}
		\begin{cases}
			\det\left(u_{i\bar{j}}\right)=1 & \text{in\  }\Omega,\\
			u=0                  & \text{on }\partial\Omega .
		\end{cases}
	\end{equation}
	\end{problem}
	It is known that the solution $u$ itself is in general not convex in $\Omega$. This leads naturally to the question of whether, for a given exponent $0<\alpha <1$, the power function $-(-u)^\alpha$ is  convex in $\Omega$. 
	In this paper, we establish the convexity of $-\sqrt{-u}$ in the case of complex dimension $2$.
	\begin{theorem}
		\label{mainTheorem}
		Let $\Omega$ be a strictly convex bounded domain in $\mathbb{C}^2$ with smooth boundary and $u\in C^4(\Omega)\cap C^2(\overline{\Omega})$ be the classical solution of (\ref{main equation}).
		Then $-\sqrt{-u}$ is strictly convex in $\Omega$.
	\end{theorem}
	
This result is also proved by a very recent paper of Zhang-Zhou \cite{ZhangZhou}, using a different method.
	
	In this paper, we employ a method developed in \cite{hu25}, \cite{hu-maxrank-2024},  \cite{Hu-metric-lower-bound24} and \cite{HuSheng}, which could be used  to study the real convexity(positivity of the real Hessian) of a solution using complex derivatives. Specifically, in this paper, we introduce a real-valued auxiliary function $\sigma$( involving only complex derivatives $u_{\ijbar}, \ u_{ij}, u_i$, where $i,j$ are indices for complex coordinates) quantifying the convexity of $-\sqrt{-u}$, so that proving the strict convexity of $-\sqrt{-u}$ is reduced  to proving the positivity of  $\sigma$. The advantage is that for an equation involving only complex derivatives $u_{\ijbar}$, using complex derivatives of $u$ is much easier than using real derivatives of $u$. 
	
	The paper is organized as follows. In Section \ref{sec:equation_for _secondDerivatives}, we derive differential equations for the second derivatives of the solution $u$ and  $-\sqrt{-u}$. 
	In Section 3, we estimate the minimal eigenvalue  of $\mm \ma^{-1}$. Section \ref{subsec:compute_SecondDerivative} and  \ref{sec:2d} treat the case where the minimal eigenvalue of  $\mm \ma^{-1}$ is strictly smaller than the second minimal eigenvalue of  $\mm \ma^{-1}$, while Section \ref{sec:generalCase_comparison} addresses the general case via approximation.

	\section{Differential Equations for Second Derivatives and Their Modifications}
	\label{sec:equation_for _secondDerivatives}
	Suppose that $u$ is a solution to equation (\ref{main equation}). In this section, we derive the important equations satisfied by the second derivatives of $u$ and $v:=-\sqrt{-u}$.
	\subsection{Equations for Second Derivatives of $u$}
	We apply $\partial_i$ to the equation of $u$:
	\begin{equation*}
		\log \left[ \det (u_{p\bar{q}})\right]= \text{constant },
	\end{equation*}
	and get
	\begin{equation}\label{Eq_1st_Derivative_2.1}
		u^{\bar{q}p}u_{p\bar{q}i}=0.
	\end{equation}
		\begin{remark}
		Throughout this paper, we adopt the convention that the first index of a matrix entry is the \emph{row} index and the second index is the \emph{column} index. For instance, we write the inverse of invertible $A$ as
		\begin{equation}
			A^{-1}=(A^{\bar{j}i}),    
		\end{equation}
		meaning that $j$ is the row index and $i$ is the column index (equivalently, $A_{k\bar{j}}\,A^{\bar{j}i}=\delta_k^{\,i}$). For the inverse of the complex conjugate matrix $\overline{A}$, we write
		\begin{equation}
			\overline{A}^{-1}=(A^{j\bar{i}}),    
		\end{equation}
		where $j$ is the row index and $i$ is the column index; in particular, $A^{j\bar{i}}=\overline{A^{\bar{j}i}}$.
		Moreover, if $A$ is Hermitian, then its inverse is also Hermitian. In particular, at the level of entries we have $A^{\bar{j}i}=A^{i\bar{j}}$
	\end{remark}
	Then we apply $\partial_{\bar{j}}$ to equation (\ref{Eq_1st_Derivative_2.1}) and get 
	\begin{equation}\label{2_1_2}
		u^{\bar{q}p}u_{p\bar{q}i\bar{j}}-u^{\bar{q}s}u_{s\bar{t}\bar{j}}u^{\bar{t}p}u_{p\bar{q}i}=0.
	\end{equation}
	We switch indices $p$ and $s$ in the second term of equation (\ref{2_1_2}) and get 
	\begin{equation}\label{2_1_3}
		u^{\bar{q}p}u_{p\bar{q}i\bar{j}}-u^{\bar{q}p}u_{p\bar{t}\bar{j}}u^{\bar{t}s}u_{s\bar{q}i}=0.
	\end{equation}
	We switch indices $t$ and $q$ in the second term of equation (\ref{2_1_2}) and get 
	\begin{equation}\label{2_1_4}
		u^{\bar{q}p}u_{p\bar{q}i\bar{j}}-u^{\bar{t}s}u_{s\bar{q}\bar{j}}u^{\bar{q}p}u_{p\bar{t}i}=0.
	\end{equation}
	Let 
	\begin{equation*}\label{def of A,B}
		A=(A_{i\bar{j}})=(u_{i\bar{j}}),\quad B=(B_{ij})=(u_{ij}),
	\end{equation*}
	where $i$ is the row index and $j$ is the column index. 

	Since $u$ is strictly plurisubharmonic, the  Hermitian matrix $A$ is positive definite. It is convenient to write equations (\ref{2_1_3}) and (\ref{2_1_4}) as equations for matrices.  
	Here, let $i$ be the row index and $j$ be the column index in equations (\ref{2_1_3}) and (\ref{2_1_4}). Then equation (\ref{2_1_3}) becomes 
	\begin{equation}\label{Matrix_A_Equation_BB}
		u^{\bar{q}p}\partial_{p\bar{q}}A=u^{\bar{q}p}\partial_{\bar{q}}B\overline{A}^{-1}\partial_p \overline{B},
	\end{equation}
	and equation (\ref{2_1_4}) becomes
	\begin{equation}\label{Matrix_A_Equation_AA}
		u^{\bar{q}p}\partial_{p\bar{q}}A=u^{\bar{q}p}\partial_pA A^{-1}\partial_{\bar{q}}A.
	\end{equation}
	
	Similarly, we apply $\partial_j$ to equation (\ref{Eq_1st_Derivative_2.1}) and get 
	\begin{equation}\label{2_1_5}
		u^{\bar{q}p}u_{p\bar{q}ij}-u^{\bar{q}s}u_{s\bar{t}j}u^{\bar{t}p}u_{p\bar{q}i}=0.
	\end{equation}
	We switch indices $p$ and $s$ in the second term of equation (\ref{2_1_5}) and get 
	\begin{equation}\label{2_1_6}
		u^{\bar{q}p}u_{p\bar{q}ij}-u^{\bar{q}p}u_{p\bar{t}j}u^{\bar{t}s}u_{s\bar{q}i}=0.
	\end{equation}
	We switch indices $q$ and $t$ in the second term of equation (\ref{2_1_5}) and get 
	\begin{equation}\label{2_1_7}
		u^{\bar{q}p}u_{p\bar{q}ij}-u^{\bar{t}s}u_{s\bar{q}j}u^{\bar{q}p}u_{p\bar{t}i}=0.
	\end{equation}
	We express equations (\ref{2_1_6}) and (\ref{2_1_7}) as matrices, where $i$ and $j$ denote the row and column indices, respectively. Then equation (\ref{2_1_6}) becomes 
	\begin{equation}\label{Matrix_B_Equation_BA}
		u^{\bar{q}p}\partial_{p\bar{q}}B=u^{\bar{q}p}\partial_{\bar{q}}B \overline{A}^{-1}\partial_p \overline{A},
	\end{equation}
	and equation (\ref{2_1_7}) becomes
	\begin{equation}\label{Matrix_B_Equation_AB}
		u^{\bar{q}p}\partial_{p\bar{q}}B= u^{\bar{q}p} \partial_pA A^{-1} \partial_{\bar{q}}B.
	\end{equation}
	Notice that equations (\ref{Matrix_A_Equation_AA}), (\ref{Matrix_A_Equation_BB}), (\ref{Matrix_B_Equation_BA}) and (\ref{Matrix_B_Equation_AB}) are analogs of equations (2.2) (2.3) of \cite{hu25}.
	The following equations are equivalent to equations (\ref{Matrix_B_Equation_BA}) and (\ref{Matrix_B_Equation_AB}), and they are very useful for subsequent computations. Equation 
	\begin{equation}\label{Eq—BqAp}
		u^{\bar{q}p}\partial_p\left(\partial_{\bar{q}}B \overline{A}^{-1}\right)=0
	\end{equation}
	is equivalent to equation (\ref{Matrix_B_Equation_BA}). Equation 
	\begin{equation}\label{Eq_AbarBbarpq}
		u^{\bar{q}p}\partial_{\bar{q}}\left( \overline{A}^{-1}\partial_p\overline{B}\right)=0
	\end{equation}
	is equivalent to the conjugate of equation (\ref{Matrix_B_Equation_BA}). 
	Then we apply $u^{\bar{q}p}\partial_{p\bar{q}}$ to the matrix 
	\begin{equation}
		M= A - B\overline{A}^{-1}\overline{B}
	\end{equation}
	and get 
	\begin{equation}
		u^{\bar{q}p}\partial_{p\bar{q}}M =u^{\bar{q}p}\partial_{\bar{q}}\left( \partial_pA -\mathbb{B}^{(p)} \overline{A}^{-1}\overline{B}-B\overline{A}^{-1}\partial_p\overline{B} \right).
	\end{equation}
	In the above, for the sake of simplifying the following computation, we introduce a new matrix 
	\begin{equation}\label{2_1_8}
		\mathbb{B}^{(p)}=\p B-B\overline{A}^{-1}\partial_p\overline{A},
	\end{equation}
	such that 
	\begin{equation}\label{2-17}
		\mathbb{B}^{(p)}\overline{A}^{-1}=\partial_p\left( B\overline{A}^{-1}\right).
	\end{equation}
	Taking the conjugate transpose of equation (\ref{2-17}) yields
	\begin{equation}
		\overline{A}^{-1}\left( \mathbb{B}^{(q)}\right)^*=\q\left( \overline{A}^{-1}\overline{B}\right).
	\end{equation}
	Then 
	\begin{align}
		u^{\bar{q}p}\pq M &=u^{\bar{q}p} \left( \pq A -\q B\overline{A}^{-1}\p \overline{B}\right)  -u^{\bar{q}p}B\partial_{\bar{q}}\left( \overline{A}^{-1}\partial_p\overline{B}\right) \nonumber \\
		&\hspace{2ex} +u^{\bar{q}p}\left[-   \q\mathbb{B}^{(p)} \overline{A}^{-1}\overline{B} -\mathbb{B}^{(p)}\overline{A}^{-1}\left(\mathbb{B}^{(q)}\right)^*   \right]     
		\label{regularpqm}
	\end{align}
	Using (\ref{Matrix_A_Equation_BB}) and (\ref{Eq_AbarBbarpq}), the equation above can be simplified:
	\begin{align}\label{2_1_9}
		u^{\bar{q}p}\pq M =   -u^{\bar{q}p} \q\mathbb{B}^{(p)} \overline{A}^{-1}\overline{B}  -u^{\bar{q}p}\mathbb{B}^{(p)}\overline{A}^{-1}\left(\mathbb{B}^{(q)}\right)^*  .
	\end{align}
	We need to compute $u^{\bar{q}p} \q\mathbb{B}^{(p)} $ in \eqref{2_1_9}. Applying $u^{\bar{q}p}\partial_{\bar{q}}$ directly to  equation (\ref{2_1_8}) yields
	\begin{align*}
		u^{\bar{q}p} \q\mathbb{B}^{(p)} &=u^{\bar{q}p} \left(\pq B-\q B\overline{A}^{-1}\p\overline{A}+B\overline{A}^{-1}\q\overline{A}\overline{A}^{-1}\p\overline{A}-B\overline{A}^{-1}\pq\overline{A}\right)\nonumber\\
		&=u^{\bar{q}p} 
		\left[ \left(  \pq B-\q B\overline{A}^{-1}\p\overline{A}\right)- B\overline{A}^{-1}\left(\q\overline{A}\overline{A}^{-1}\p\overline{A}- \pq\overline{A}\right)\right].
	\end{align*}
	Using equations (\ref{Matrix_B_Equation_BA}) and (\ref{Matrix_A_Equation_AA}), we obtain $u^{\bar{q}p} \q\mathbb{B}^{(p)} =0$. 
	So we find 
	\begin{equation}\label{M_leq_0}
		u^{\bar{q}p}\partial_{p\bar{q}}M=-u^{\bar{q}p}\mathbb{B}^{(p)}\overline{A}^{-1}\left(\mathbb{B}^{(q)}\right)^* \le 0.
	\end{equation}
	Here, for an $n\times n$ matrix $M$, we write 
	\begin{equation}
		M<C \quad (\text{ or }\,M > C),  
	\end{equation}
	if all eigenvalues of $M$ are real and $< C$( or $> C$); similarly we denote $M\le C$ or $M\ge C$ when equalities are included.
	
	This inequality (\ref{M_leq_0}) implies if $M\ge 0$ (positive definite) on $\partial \Omega$, then $M>0$ in $\Omega$.
	By Lemma A.6 of \cite{hu25} or Lemma \ref{lem:determinant_lemma} of this paper,  the strict convexity of $u$ is equivalent to the conditions $A>0$ and $M>0$. So the inequality (\ref{M_leq_0}) implies  that if u is strictly convex on $\partial \Omega$, then it is strictly convex in $\Omega$.
	
	However, there are many examples showing $u$ may not be convex on  $\partial \Omega$, even $\Omega$ is a strictly convex domain. So we turn to prove the convexity of a funtion of $u$, in particular, we study the convexity of $v:=-\sqrt{-u}$.
	
	\subsection{Differential Equations of Second Derivatives of $v$}
	By the maximum principle and the boundary condition $u(x)=0$ for $x\in \partial \Omega$, we conclude that $ -u>0$ in $\Omega$.
	The Hessian matrices of $ v:= -\sqrt{-u}$ are given by
	\begin{align}
		v_{i\bar{j}}=\frac{1}{2(-u)^\frac{1}{2}}\left[ u_{i\bar{j}}+ \frac{u_i u_{\bar{j}}}{-2u} \right],\
		v_{ij}=\frac{1}{2(-u)^\frac{1}{2}}\left[ u_{ij}+\frac{u_i u_{j}}{-2u} \right] .
	\end{align}
	Let 
	\begin{equation}
		\mathcal{A}=(\mathcal{A}_{i\bar{j}})=\left(u_{i\bar{j}}+\frac{u_iu_{\bar{j}}}{-2u}\right),\quad  
		\mathcal{B}=(\mathcal{B}_{ij})=\left(u_{ij}+\frac{u_iu_j}{-2u}\right),
	\end{equation}
	where $i$ is the row index and $j$ is the column index.  Then the matrix $\mathcal{A}$ is positive definite. Let 
	\begin{equation}
		\mathcal{M}=  \mathcal{A}- \mathcal{B} \overline{ \mathcal{A}}^{-1}\overline{ \mathcal{B}}.
	\end{equation}
	By Lemma A.6 of \cite{hu25} or Lemma \ref{lem:determinant_lemma} of this paper, the convexity of $v$ is equivalent to the positive definiteness of both $\mathcal{A}$ and $ \mathcal{M}$.

	Let 
	\begin{equation}
		\Omega_\delta=\{x\in \Omega \;|\; \text{dist} (x, \partial \Omega) > \delta \}.
	\end{equation} 
	According to the $C^2$ estimate for Problem \ref{prob:Dirichlet}, the minimum eigenvalue of $(u_{\ijbar})$ has a positive lower bound estimate in $\Omega$, depending on the convexity and regularity of $\Omega$. Since $\MA>A=(u_{\ijbar})$, the minimum eigenvalue of $\MA$ also has a positive lower bound  in $\Omega$.
	 If $\Omega$ is smooth and strictly convex, then it's well known that $v$ is strictly convex in $\Omega_\delta \setminus \Omega_{2\delta}$ when $\delta$ is small enough. Hence, $\mathcal{M} >0 $ in $\Omega_{\delta}\setminus \Omega_{2\delta}$.
	We need to show that this is also true in $\Omega_{2\delta}$. To do this, let 
	\begin{align}
		\MK=\MM\MA^{-1},
	\end{align} and we aim to prove that  the minimum eigenvalue of $\MK>0$.
	We compute and find 
	\begin{align}
		u^{\bar{q}p}( \lambda_{\min} 	\MK)_{p\bar{q}} \le C(\lambda_{\min} 	\MK + |\nabla \lambda_{\min}	\MK|),
	\end{align}
	 when the minimum eigenvalue of $\MK$ is different from the second minimal eigenvalue of $\MK$ in section \ref{sec:Estimate_Sigma} in the case of complex dimension 2. Then we can use this to derive a positive lower bound for  the minimum eigenvalue  of $\MK$.
	
	In this section, we derive the differential relations for $\mathcal{A}$  and $\mathcal{B}$, analogous to 
	the relations (\ref{Matrix_A_Equation_BB}), (\ref{Matrix_A_Equation_AA}), (\ref{Matrix_B_Equation_AB}) and (\ref{Matrix_B_Equation_BA}).
	
	In the computation for $u^{\bar{q}p}\pq M \le 0$, we used the following relations 
	\begin{itemize}
		\item  $u^{\bar{q}p}\left( \partial_{p\bar{q}}A-\partial_{\bar{q}}B\overline{A}^{-1}\partial_p \overline{B}\right)=0$,
		\item   $ u^{\bar{q}p}   \left( \pq\overline{A}-\q\overline{A}\overline{A}^{-1}\p\overline{A}\right) =0$,
		\item   $u^{\bar{q}p} \left(  \pq B-\q B\overline{A}^{-1}\p\overline{A}\right)=0$,
		\item   $ u^{\bar{q}p} \left(\pq \overline{B}-\q \overline{A}\overline{A}^{-1}\p\overline{B}\right)=0$.
	\end{itemize}
	When we compute $	u^{\bar{q}p}( \lambda_{\min} 	\MK)$, we find it's deeply related to computing 
	\begin{align}
		u^{\bar{q}p}\pq \MM.
	\end{align}
	 In the computation of $u^{\bar{q}p}\pq \mm$, the four relations above still appear but with some modifications.
		To simplify the computation, we introduce 
	\begin{equation}
		\mathscr{B}^{(p)}=\p \mb -\mb \overline{\ma}^{-1}\p \overline{\ma}.
		\label{Notation_sb}
	\end{equation}
	With this notation, we have
	\begin{equation}\label{2.40}
		\mathscr{B}^{(p)}\overline{\ma}^{-1}=\p \left( \mb \overline{\ma}^{-1} \right).
	\end{equation}
	Taking the conjugate transpose of equation \eqref{2.40} yields
	\begin{equation}
		\overline{\ma}^{-1}(\mathscr{B}^{(q)})^*=\q ( \overline{\ma}^{-1} \overline{\mb}). 
	\end{equation}
  First we compute $\partial_p\MM$:
  \begin{align}
  	\partial_p\MM=\partial_p\MA-\mathscr{B}^{( p)}{\overline\MA}^{-1}\overline{\MB}-\MB{\overline\MA}^{-1}\partial_p{\overline\MB}.
  \end{align}
  Then we apply $\partial_\qbar$ to the equation above. Simply using Leibniz rule, we get
	\begin{align}
		&u^{\bar{q}p}\pq \mm =u^{\bar{q}p} \left( \pq \ma -\q \mb\overline{\ma}^{-1}\p \overline{\mb}\right)  -u^{\bar{q}p}\mb\partial_{\bar{q}}\left( \overline{\ma}^{-1}\partial_p\overline{\mb}\right)  \nonumber\\
		&\hspace{2ex} +u^{\bar{q}p}\left[-   \q\mathscr{B}^{(p)} \overline{\MA}^{-1}\overline{\mb} -\mathscr{B}^{(p)}\overline{\MA}^{-1}\left(\partial_q\mathscr{B}\right)^*   \right] ,
		\label{pqM}
		\end{align}
		which is very similar to (\ref{regularpqm}).
		In above 
		\begin{align}
			\partial_\qbar\left(\overline{\MA}^{-1}\partial_p\overline\MB\right)=-\MAbarinverse
			\left(\partial_{\qbar} \MAbar\MAbarinverse\partial_p\overline\MB-\partial_{\pqbar}\MBbar
			\right)
			\label{qpB}
		\end{align}
		and 
		\begin{align}
			\label{qsBp}
		 \q\mathscr{B}^{(p)}&	=\partial_{\qbar} \left(
		 \partial_p\MB-\MB\MAbarinverse\partial_p\MAbar
		 \right)\nonumber\\
		 =&\left(
		 \partial_\pqbar\MB-\partial_\qbar\MB\MAbarinverse\partial_p\MAbarinverse
		 \right)-\MB\MAbarinverse\left(\pq \overline{\ma}-\q \overline{\ma}\overline{\ma}^{-1}\p \overline{\ma}
		 \right).
		\end{align}
		Plugging (\ref{qpB}) and (\ref{qsBp}) into (\ref{pqM}), we get
		\begin{align}
	&\ \ 	u^{\bar{q}p}\pq \mm\nonumber \\&=u^{\bar{q}p}\left( \pq \ma -\q \mb \overline{\ma}^{-1}\p \overline{\mb}\right)+ \mb\overline{\ma}^{-1} u^{\bar{q}p}\left(\pq \overline{\ma}-\q \overline{\ma}\overline{\ma}^{-1}\p \overline{\ma}\right)\overline{\ma}^{-1}\overline{\mb} \nonumber\\
		&\hspace{2ex} - u^{\bar{q}p}\left( \pq \mb- \q \mb \overline{\ma}^{-1}\p \overline{\ma}\right)\overline{\ma}^{-1}\overline{\mb}-\mb \overline{\ma}^{-1}u^{\bar{q}p}\left( \pq \overline{\mb}-\q \overline{\ma}\overline{\ma}^{-1}\p\overline{\mb} \right)\nonumber\\
		&\hspace{2ex}  - u^{\bar{q}p}\mathscr{B}^{(p)}\overline{\ma}^{-1}\left( \mathscr{B}^{(q)}\right)^{*} .\label{upqmpq}
	\end{align}

	Similar to the computation for $u^{\bar{q}p}\pq M$, evaluating \eqref{upqmpq}  requires to compute the following relations:
	\begin{itemize}
		\item $u^{\bar{q}p} \left( \pq \ma -\q \mb\overline{\ma}^{-1}\p \overline{\mb}\right) $,
		\item $u^{\bar{q}p} \left(\partial_{p\bar{q}}\overline{\ma} -\partial_\qbar \overline{\ma} \overline{\ma}^{-1}\partial_{p}\overline{\ma}\right) $,
		\item $u^{\bar{q}p}\left(\partial_{p\bar{q}}\mb-\partial_{\bar{q}}\mb \overline{\ma}^{-1}\partial_p \overline{\ma}\right) $,
		\item $u^{\bar{q}p} \left(\pq \overline{\mb}-\q \overline{\ma}\overline{\ma}^{-1}\p\overline{\mb}\right)$.
	\end{itemize}
	
	We first compute the derivatives of $\ma$ and $\mb$. It is convenient to define $h(u) =\frac{1}{-2u}$, whose derivative is $h'=2h^2$.
	By direct computation, we get 
	\begin{align}
		\p \ma_{i\bar{j}}&= \p A_{i\bar{j}}+u_{ip}u_{\bar{j}}h+u_iu_{p\bar{j}}h+u_iu_{\bar{j}}u_ph'
		\nonumber\\
		&=\p A_{i\bar{j}}+hu_{\bar{j}}\mb_{ip}+hu_i\ma_{p\bar{j}}.\label{2_2_10_1}
	\end{align}
	Similarly, we have 
	\begin{align}
		\q \ma_{i\bar{j}}&=\q A_{i\bar{j}}+hu_{\bar{j}}\ma_{i\bar{q}}+hu_i\overline{\mb_{qj}},\label{2_2_10_2}\\
		\p \mb_{ij}&=\p B_{ij}+hu_{j}\mb_{ip}+hu_i\mb_{pj}             ,\label{2_2_10_3}\\
		\q \mb_{ij}&=\q B_{ij}+hu_j \ma_{i\bar{q}}+hu_i\ma_{j\bar{q}},\label{2_2_10_4}\\
		\p \mb_{\bar{i}\bar{j}}&=\p B_{\bar{i}\bar{j}}+hu_{\bar{j}}\ma_{p\bar{i}}+hu_{\bar{i}}\ma_{p\bar{j}}.
	\end{align}
	We plug these into 
	\begin{equation}
		\p \mm =\p \ma -\mathscr{B}^{(p)} \overline{\MA}^{-1}\overline{\mb}-\mb \overline{\ma}^{-1}\p\overline{\mb}.
	\end{equation}
	In the case that $\ma=I$ and $\mb=\mathrm{diag}(b_1,\cdots,b_n)$, for $b_k\in \mathbb{R}^{\ge 0}$, we get
	\begin{align}
		&\hspace{2ex} \p \mm_{i\bar{j}}\nonumber \\
		&=u_{i\bar{j}p}+hu_{\bar{j}}\mb_{ip}+hu_i\ma_{p\bar{j}}-b_j \mathscr{B}^{(p)}_{i{j}}-b_iu_{\bar{i}\bar{j}p}-hb_iu_{\bar{j}}\ma_{p\bar{i}}-hb_iu_{\bar{i}}\ma_{p\bar{j}}\nonumber\\
		&=u_{i\bar{j}p}-b_iu_{\bar{i}\bar{j}p}+hu_i\delta_{jp}-hb_iu_{\bar{i}}\delta_{jp}-b_j \mathscr{B}^{(p)}_{i{j}}.\label{naivepartialp_ijbar}
	\end{align}
	When $i=j=1$, we get
	\begin{equation}
		\p \mm_{1\bar{1}}=u_{1\bar{1}p}-b_1u_{\bar{1}\bar{1}p}+hu_1\delta_{1p}-hb_1u_{\bar{1}}\delta_{1p}-b_1 \mathscr{B}^{(p)}_{11}.
	\end{equation}
	
	Then we compute the second derivatives of $\ma_{i\bar{j}}$
	\begin{align}
		\pq \ma_{i\bar{j}}&=\pq A_{i\bar{j}}+h'u_{\bar{q}}u_{\bar{j}}\mb_{ip}+hu_{\bar{j}\bar{q}}\mb_{ip}+hu_{\bar{j}}\q \mb_{ip}\nonumber\\
		&\hspace{2ex}+h'u_{\bar{q}}u_i\ma_{p\bar{j}}+hu_{i\bar{q}}\ma_{p\bar{j}}+hu_i\q\ma_{p\bar{j}}.
	\end{align}
	We use $h'=2h^2$, $u_{\bar{j}\bar{q}}+hu_{\bar{j}}u_{\bar{q}}=\overline{\mb_{jq}}$ and $u_{i\bar{q}}+hu_iu_{\bar{q}}=\ma_{i\bar{q}}$, and get 
	\begin{align}
		&\pq \ma_{i\bar{j}}=\pq A_{i\bar{j}}+hu_{\bar{j}}\q B_{ip}+hu_i\q A_{p\bar{j}}+h\mb_{ip}\overline{\mb_{jq}}+h\ma_{i\bar{q}}\ma_{p\bar{j}} \nonumber\\
		&+h^2u_pu_{\bar{j}}\ma_{i\bar{q}}+2h^2u_iu_{\bar{j}}\ma_{p\bar{q}}+h^2u_iu_p\overline{\mb_{qj}}+h^2u_{\bar{j}}u_{\bar{q}}\mb_{ip}+h^2u_iu_{\bar{q}}\ma_{p\bar{j}}.\label{2_2_12}
	\end{align}
	As an analog of (\ref{Matrix_A_Equation_BB}), we need to compute 
	\begin{equation}\label{2_2_12_2}
		u^{\bar{q}p} \pq \ma - u^{\bar{q}p}\q \mb\overline{\ma}^{-1}\p \overline{\mb}
	\end{equation} since it appears in (\ref{upqmpq}). 
	When computing $	u^{\bar{q}p} \pq \ma$, we notice that on the right-hand side of (\ref{2_2_12}) 
	\begin{align}
  u^{\pqbar}(	hu_{\bar{j}}\q B_{ip}+hu_i\q A_{p\bar{j}})	=0
	\end{align}
	since $u^{\pqbar}	\q B_{ip}=u^{\pqbar}\q A_{p\bar{j}}	=0$. Then combining (\ref{2_2_12}) and (\ref{2_2_10_4}) yields 
	\begin{align}
		&\left( u^{\bar{q}p} \pq \ma - u^{\bar{q}p}\q \mb\overline{\ma}^{-1}\p \overline{\mb} \right)_{i\bar{j}}
		\nonumber\\
		&= u^{\bar{q}p} \pq \ma_{i\bar{j}}-u^{\bar{q}p} (\q \mb)_{is}\ma^{s\bar{t}}(\p \overline{\mb} )_{\bar{t}\bar{j}}\nonumber\\
		&=u^{\bar{q}p}u_{is\bar{q}}(A^{s\bar{t}}-\ma^{s\bar{t}})u_{\bar{j}\bar{t}p}+hu^{\bar{q}p}(\mb_{ip}\mb_{\bar{j}\bar{q}}+\ma_{i\bar{q}}\ma_{p\bar{j}})\nonumber\\
		&\hspace{2ex}-hu^{\bar{q}p}\ma^{s\bar{t}}(u_{\bar{t}}u_{is\bar{q}}\ma_{p\bar{j}} +u_su_{\bar{j}\bar{t}p}\ma_{i\bar{q}}) \nonumber\\
		&\hspace{2ex}+h^2u^{\bar{q}p}(u_iu_{\bar{j}}\ma_{p\bar{q}}+u_iu_p\mb_{\bar{j}\bar{q}}+u_{\bar{j}}u_{\bar{q}}\mb_{ip}-u_su_{\bar{t}}\ma^{s\bar{t}}\ma_{i\bar{q}}\ma_{p\bar{j}}).
	\end{align}
	Similar to  (\ref{Matrix_A_Equation_AA}), we want to compute
	\begin{equation}
		u^{\bar{q}p}\partial_{p\bar{q}}\overline{\ma} -u^{\bar{q}p}\partial_\qbar \overline{\ma} \overline{\ma}^{-1}\partial_{p}\overline{\ma}.
	\end{equation}
	Combining (\ref{2_2_12}), (\ref{2_2_10_1}) and (\ref{2_2_10_2}) yields
	\begin{align}
		&\left(   u^{\bar{q}p}\partial_{p\bar{q}}\overline{\ma} -u^{\bar{q}p}\partial_\qbar \overline{\ma} \overline{\ma}^{-1}\partial_{p}\overline{\ma} \right)_{\bar{i}j}\nonumber \\
		&= u^{\bar{q}p}\partial_{p\bar{q}}\overline{\ma}_{\bar{i}j}-u^{\bar{q}p}(\partial_\qbar\overline{\ma} )_{\bar{i}s}\ma^{s\bar{t}}(\partial_{p}\overline{\ma} )_{\bar{t}j}\nonumber\\
		&=u^{\bar{q}p}u_{s\bar{i}\bar{q}}(A^{s\bar{t}}-\ma^{s\bar{t}}) u_{j\bar{t}p}+hu^{\bar{q}p}(\mb_{jp}\mb_{\bar{i}\bar{q}}+\ma_{j\bar{q}}\ma_{p\bar{i}})\nonumber\\
		&\hspace{2ex} -hu^{\bar{q}p}\ma^{s\bar{t}}(u_{\bar{t}}u_{s\bar{i}\bar{q}}\mb_{jp}+u_su_{j\bar{t}p}\mb_{\bar{i}\bar{q}})\nonumber   \\
		&\hspace{2ex}+h^2u^{\bar{q}p}(u_ju_{\bar{i}}\ma_{p\bar{q}}+u_pu_{\bar{i}}\ma_{j\bar{q}}+u_ju_{\bar{q}}\ma_{p\bar{i}}-u_su_{\bar{t}}\ma^{s\bar{t}}\mb_{jp}\mb_{\bar{i}\bar{q}}).
	\end{align}
	Also, we need to compute 
	\begin{equation}\label{2_2_15}
		u^{\bar{q}p}\partial_{p\bar{q}}\mb-u^{\bar{q}p}\partial_{\bar{q}}\mb \overline{\ma}^{-1}\partial_p \overline{\ma}.
	\end{equation}
	Then, we compute the second derivatives of $\mb_{ij}$
	\begin{align}
		&\pq \mb_{ij}=\p \left( \q B_{ij}+hu_j \ma_{i\bar{q}}+hu_i\ma_{j\bar{q}} \right)\nonumber\\
		&=\pq B_{ij}+hu_j\p A_{i\bar{q}}+hu_i\p A_{j\bar{q}}+h\ma_{i\bar{q}}\mb_{jp}+h\ma_{j\bar{q}}\mb_{ip} \nonumber\\
		&+h^2u_pu_j\ma_{i\bar{q}}+h^2u_iu_p\ma_{j\bar{q}}+2h^2u_iu_j\ma_{p\bar{q}}+h^2u_ju_{\bar{q}}\mb_{ip}+h^2u_iu_{\bar{q}}\mb_{jp}.
	\end{align}
	It follows that 
	\begin{align}
		&  \left(  u^{\bar{q}p}\partial_{p\bar{q}}\mb-u^{\bar{q}p}\partial_{\bar{q}}\mb \overline{\ma}^{-1}\partial_p \overline{\ma} \right)_{ij}\nonumber\\
		&=u^{\bar{q}p}\pq \mb_{ij} -u^{\bar{q}p} (\partial_{\bar{q}}\mb)_{is}\ma^{s\bar{t}}(\partial_p \overline{\ma} )_{\bar{t}j}\nonumber\\
		&=u^{\bar{q}p}u_{is\bar{q}}(A^{s\bar{t}}-\ma^{s\bar{t}})u_{j\bar{t}p} +hu^{\bar{q}p} ( \ma_{i\bar{q}}\mb_{jp}+\ma_{j\bar{q}}\mb_{ip}) \nonumber\\
		&\hspace{2ex}-hu^{\bar{q}p}\ma^{s\bar{t}} (u_{\bar{t}}u_{is\bar{q}}\mb_{pj}+u_su_{j\bar{t}p}\ma_{i\bar{q}})\nonumber\\
		&\hspace{2ex} +h^2u^{\bar{q}p}(u_iu_j\ma_{p\bar{q}}+u_iu_p\ma_{j\bar{q}}+u_ju_{\bar{q}}\mb_{ip}-u_su_{\bar{t}}\ma^{s\bar{t}}\ma_{i\bar{q}}\mb_{jp}).
	\end{align}
	It remains to compute the last term,
	\begin{equation}
		\pq \overline{\mb}-\q \overline{\ma}\overline{\ma}^{-1}\p\overline{\mb}. 
	\end{equation}
	Then, we  compute the second derivatives of $\mb_{\bar{i}\bar{j}}$
	\begin{align}
		&\pq \mb_{\bar{i}\bar{j}} =\p (\q B_{\bar{i}\bar{j}}+hu_{\bar{i}}\mb_{\bar{q}\bar{j}}+hu_{\bar{j}}\mb_{\bar{q}\bar{i}})\nonumber\\
		&=\pq B_{\bar{i}\bar{j}}+hu_{\bar{i}}\p B_{\bar{j}\bar{q}}+hu_{\bar{j}}\p B_{\bar{i}\bar{q}}+h\ma_{p\bar{i}}\mb_{\bar{j}\bar{q}}+h\ma_{p\bar{j}}\mb_{\bar{q}\bar{i}}\nonumber\\
		&+h^2u_pu_{\bar{i}}\mb_{\bar{j}\bar{q}}+h^2u_pu_{\bar{j}}\mb_{\bar{i}\bar{q}}+2h^2u_{\bar{i}}u_{\bar{j}}\ma_{p\bar{q}}+h^2u_{\bar{j}}u_{\bar{q}}\ma_{p\bar{i}}+h^2u_{\bar{i}}u_{\bar{q}}\ma_{p\bar{j}}.
	\end{align}
	It follows that 
	\begin{align}
		& \left( u^{\bar{q}p} \pq \overline{\mb}-u^{\bar{q}p}\q \overline{\ma}\overline{\ma}^{-1}\p\overline{\mb}\right)_{\bar{i}\bar{j}}\nonumber\\
		&=u^{\bar{q}p}\pq \mb_{\bar{i}\bar{j}}- u^{\bar{q}p} (\q\overline{\ma})_{\bar{i}s} \ma^{s\bar{t}}(\p\overline{\mb})_{\bar{t}\bar{j}}\nonumber\\
		&=u^{\bar{q}p}u_{\bar{i}s\bar{q}}(A^{s\bar{t}}-\ma^{s\bar{t}})u_{\bar{j}\bar{t}p} +hu^{\bar{q}p} (\ma_{p\bar{i}}\mb_{\bar{j}\bar{q}}+ \ma_{p\bar{j}}\mb_{\bar{q}\bar{i}})\nonumber\\
		&\hspace{2ex}-hu^{\bar{q}p}\ma^{s\bar{t}}(u_{\bar{t}}u_{\bar{i}s\bar{q}}\ma_{p\bar{j}}+u_su_{\bar{j}\bar{t}p}\mb_{\bar{i}\bar{q}})\nonumber \\
		&\hspace{2ex}+h^2u^{\bar{q}p}( u_{\bar{i}}u_{\bar{j}}\ma_{p\bar{q}}+u_pu_{\bar{i}}\mb_{\bar{j}\bar{q}}+u_{\bar{q}}u_{\bar{j}}\ma_{p\bar{i}}-u_s u_{\bar{t}}\ma^{s\bar{t}}\mb_{\bar{i}\bar{q}}\ma_{p\bar{j}} ).
	\end{align}

	\section{Estimate for Minimal Eigenvalue of $\mm \ma^{-1}$}
	\label{sec:Estimate_Sigma}
	Denote
	\begin{align}
		\MK=\mm \ma^{-1}
	\end{align}
and denote the minimum eigenvalue of the matrix $\mm \ma^{-1}$ by $\sigma$.
We will find a lower bound estimate for $\sigma$ in this section. 

In subsection \ref{subsec:compute_SecondDerivative}, we compute $	u^{\pqbar}\sigma_{\pqbar}$ when the minimum eigenvalue of $\MK$ is different from the second minimum eigenvalue of $\MK$. We reduce the result to a quadratic form.

In subsection \ref{sec:2d}, we prove the non-positivity of the quadratic form above in the case of complex dimension 2; also we show that the sign could be negative in dimension 3. This is the only place where we require that the dimension is 2. This implies that $\sigma$ satisfies the following equation:
\begin{align}
	u^{\pqbar}\sigma_{\pqbar}\leq C(\sigma+|\nabla \sigma|)
\end{align}for a constant $C$ depending on the solution $u$ when the minimum eigenvalue of $\MK$ is different from the second minimum eigenvalue of $\MK$.

But the computation above is only available when the minimum eigenvalue of $\MK$ is different from the second minimum eigenvalue of $\MK$. In subsection \ref{sec:generalCase_comparison}, using a perturbation method, we prove the general situation.

	\subsection{Computation When the Minimal Eigenvalue of $\mm \ma^{-1}$ is Strictly Smaller Than the Second Minimal Eigenvalue of $\mm \ma^{-1}$ }
	\label{subsec:compute_SecondDerivative}
	In this section, we assume that the minimal eigenvalue of $\mm \ma^{-1}$ is strictly smaller than the second minimal eigenvalue. In this setting, the minimal eigenvalue $\sigma$ is a smooth function. We need to utilize the formula derived in Section \ref{sec:equation_for _secondDerivatives} and Lemma \ref{lem:derivates_min_eig}.
	
	We assume that, at the point ${\bf{z}}$ where the computation is carried out,
	\begin{equation}
		\ma =I, \text{ and  }  \mb= \mathrm{diag}(b_1,\cdots, b_n), \text{ with } b_i\in \mathbb{R}^{\ge 0}.
	\end{equation}
	This is always possible: we first choose local coordinates so that $\ma =I$, and then diagonalize $\mb$ by a unitary change of coordinates (see Lemma A.3 in \cite{hu25} or Corollary 4.4.4(c) in \cite{Horn-Johnson2013}).
	With this coordinate, we get 
	\begin{equation}
		\mm=I-\mb \mb^*=\mathrm{diag}\left( 1-b_1^2, \cdots, 1-b_n^2 \right)=\mathrm{diag}\left(\mu_1,\cdots,\mu_n\right),
	\end{equation}
	where we set $\mu_i = 1-b_i^2$. We assume $b_1$ is the maximum eigenvalue of $\mb$. Then the minimal eigenvalue of $\mm$ at ${\bf{z}}$ is 
	\begin{equation}
		\mu_1= (1-b_1^2).
	\end{equation}
	Also, we assume $\mu_1\ge 0$. In this section, we show, with a constant $C$,
	\begin{equation}
		u^{\bar{q}p}\pq\sigma\le C(\sigma+|\nabla \sigma|).
	\end{equation}
	Following the notation in \cite{Caffarelli-Friedman1985}, which is also adopted by \cite{Caffarelli-Guan-Ma2007} and \cite{Bian-Guan-Ma-Xu2011}, for two functions $f$, $h$, defined in an open set $S\subset \mathbb{C}^n$, and a point $y\in S$, we say 
	\begin{itemize}
		\item $(f - h )(y)\lesssim 0$, if $(f-h)(y) \le C\cdot\left[ \sigma (y) + |\nabla \sigma(y)| \right]$.
		\item $(f - h)(y)\sim 0$ or $f(y)\sim h(y)$, if $(f-h)(y) \lesssim 0$ and $(h-f)(y) \lesssim 0$.
		\item  $f-h \lesssim 0$, if $(f-h)(y) \lesssim 0$ for all $y\in S$.
		\item  $(f-h) \sim 0$ or $f\sim h$, if $(f-h)(y)\sim 0$ for all $y\in S$.
	\end{itemize}
	In this section, we aim to prove that
	\begin{equation}
		u^{\bar{q}p}\pq\sigma \lesssim 0.
	\end{equation}
	In what follows, all computations are carried out at the point ${\bf{z}}$. To simplify notation, whenever there is no ambiguity we omit the explicit dependence on ${\bf{z}}$: we write $(f-h)(z)\lesssim 0$ in place of $(f-h)\lesssim 0$, and $(f-h)(z)\sim 0$ in place of $(f-h)\sim 0$.
	Using Lemma B, we get 
	\begin{equation}\label{3.2-1}
		\pq\sigma= \pq (\mm \ma^{-1})_1^1-\sum_{s\neq1}\frac{\p (\mm \ma^{-1})_1^s \q (\mm \ma^{-1})_s^1}{\mu_s-\mu_1}
		-\sum_{s\neq 1}\frac{\q (\mm \ma^{-1})_1^s \p (\mm \ma^{-1})_s^1}{\mu_s-\mu_1}.
	\end{equation}
	We first compute $\p (\mm \ma^{-1})_1^s$ for $s\neq 1$ in the equation above,
	\begin{equation}
		\p (\mm \ma^{-1})_1^s=\left[ (\p \mm)\ma^{-1} -\mm \ma^{-1}\p \ma \ma^{-1} \right]_1^s.
	\end{equation}
	Using the assumption that $\mm$ is diagonal and $\ma=I$, the above expression can be simplified as 
	\begin{equation}
		\p (\mm \ma^{-1})_1^s({{\bf{z}}}) = \p \mm_{1\bar{s}}-\mu_1 \p \ma_{1\bar{s}}.
	\end{equation}
	Since $\sigma ({\bf{z}})=\mu_1$, we have $\mu_1 \sim 0$. Therefore, we get
	\begin{equation}\label{3.3-1}
		\p (\mm \ma^{-1})_1^s \sim \p \mm_{{1\bar{s}}}.
	\end{equation}
	Similarly, for $s\neq 1$, we get 
	\begin{equation}\label{3.3-2}
		\q (\mm \ma^{-1})_s^1({\bf{z}}) = \q \mm_{s\bar{1}}-\mu_s \q \ma_{s\bar{1}}.
	\end{equation}
	We plug (\ref{3.3-1}) and (\ref{3.3-2}) into the third term of (\ref{3.2-1}) to simplify it and get 
	\begin{equation}\label{3.4.1}
		\frac{\p (\mm \ma^{-1})_1^s \q (\mm \ma^{-1})_s^1}{\mu_s-\mu_1}\sim \frac{\p \mm_{{1\bar{s}}}\left(  \q \mm_{s\bar{1}}-\mu_s \q \ma_{s\bar{1}} \right)}{\mu_s-\mu_1}.
	\end{equation}
	Since $\mu_s >\mu_1 \ge 0$, we get $\frac{1}{\mu_s-\mu_1} \sim \frac{1}{\mu_s}$.  Then 
	\begin{equation}\label{3.4.2}
		\frac{\p (\mm \ma^{-1})_1^s \q (\mm \ma^{-1})_s^1}{\mu_s-\mu_1}\sim \left( \frac{\p \mm_{{1\bar{s}}}\q \mm_{s\bar{1}}}{\mu_s}-\p \mm_{{1\bar{s}}} \q \ma_{s\bar{1}} \right).
	\end{equation}
	To simplify the last term of (\ref{3.2-1}), for $s\neq 1$, we have
	\begin{align}
		\q (\mm \ma^{-1})_1^s &= \left[ (\q \mm)\ma^{-1} -\mm \ma^{-1}\q \ma \ma^{-1} \right]_1^s\\
		&=\left(\q \mm_{1\bar{s}}-\mu_1 \p \ma_{1\bar{s}}\right)\sim \q \mm_{1\bar{s}},\\
		\text{ and }\quad \p (\mm \ma^{-1})_s^1&=\left[ (\p \mm)\ma^{-1} -\mm \ma^{-1}\p \ma \ma^{-1} \right]_s^1 \\
		&= \p \mm_{s\bar{1}}-\mu_s \p \ma_{s\bar{1}}.
	\end{align}
	Similarly, it follows that 
	\begin{equation}\label{3.4.3}
		\frac{\q (\mm \ma^{-1})_1^s \p (\mm \ma^{-1})_s^1}{\mu_s-\mu_1}\sim \left( \frac{\q \mm_{{1\bar{s}}}\p \mm_{s\bar{1}}}{\mu_s}-\q \mm_{{1\bar{s}}} \p \ma_{s\bar{1}} \right).
	\end{equation}
	Then we compute the second term in (\ref{3.2-1}), 
	\begin{align}
		&\pq (\mm \ma^{-1})_1^1=\q \left[ (\p \mm)\ma^{-1} -\mm \ma^{-1}\p \ma \ma^{-1} \right]_1^1 \nonumber\\
		&\sim(\pq \mm)_{1\bar{1}}-\sum_{s\neq1}(\p \mm)_{1\bar{s}}(\q \ma)_{s\bar{1}}-\sum_{s\neq1}(\q \mm)_{1\bar{s}}(\p \ma)_{s\bar{1}}.\label{3.5-1}
	\end{align}
	Substituting \eqref{3.4.2}, \eqref{3.4.3}, and \eqref{3.5-1} into \eqref{3.2-1}, we see that the last term in \eqref{3.5-1} cancels with the last term in \eqref{3.4.3}, and the last term in \eqref{3.4.2} cancels with the second-to-last term in \eqref{3.5-1}.
	Then (\ref{3.2-1}) becomes 
	\begin{equation}\label{3.5-2}
		\pq \sigma \sim (\pq \mm)_{1\bar{1}} - \sum_{s\neq 1}\frac{\p \mm_{{1\bar{s}}}\q \mm_{s\bar{1}}}{\mu_s}- \sum_{s\neq1}\frac{\q \mm_{{1\bar{s}}}\p \mm_{s\bar{1}}}{\mu_s} .
	\end{equation}
	Then we use the result of Section 2 to compute $(\pq \mm )_{1\bar{1}}$,
	\begin{align*}
		& (\pq \mm)_{1\bar{1}} = \left[\pq \left( \ma-\mb \overline{\ma}^{-1}\overline{\mb}\right)  \right]_{1\bar{1}}\\
		&=\left( \pq \ma -\q \mb \overline{\ma}^{-1}\p \overline{\mb}\right)_{1\bar{1}}+\left[ \mb\overline{\ma}^{-1} \left(\pq \overline{\ma}-\q \overline{\ma}\overline{\ma}^{-1}\p \overline{\ma}\right)\overline{\ma}^{-1}\overline{\mb}\right]_{1\bar{1}} \\
		&-\left[ \left( \pq \mb- \q \mb \overline{\ma}^{-1}\p \overline{\ma}\right)\overline{\ma}^{-1}\overline{\mb}\right]_{1\bar{1}}-\left[\mb \overline{\ma}^{-1}\left( \pq \overline{\mb}-\q \overline{\ma}\overline{\ma}^{-1}\p\overline{\mb} \right)\right]_{1\bar{1}}\\
		& - \left[\mathscr{B}^{(p)}\overline{\ma}^{-1}\left( \mathscr{B}^{(q)}\right)^{*}\right]_{1\bar{1}}.
	\end{align*}
	Using the assumption that $\ma =I$ and $\mb =\mathrm{diag}(b_1,\cdots, b_n)$, and $b_1\sim 1$, we find 
	\begin{align}
		u^{\bar{q}p}\sigma_{p\bar{q}} 
		&\sim u^{\bar{q}p}\left( \pq \ma -\q \mb \overline{\ma}^{-1}\p \overline{\mb}\right)_{1\bar{1}}+u^{\bar{q}p}\left(\pq \overline{\ma}-\q \overline{\ma}\overline{\ma}^{-1}\p \overline{\ma}\right)_{1\bar{1}} \label{3.19}\\
		&\hspace{1ex}-u^{\bar{q}p}\left( \pq \mb- \q \mb \overline{\ma}^{-1}\p \overline{\ma}\right)_{1\bar{1}}-u^{\bar{q}p}\left( \pq \overline{\mb}-\q \overline{\ma}\overline{\ma}^{-1}\p\overline{\mb} \right)_{1\bar{1}}\label{3.20}\\
		&\hspace{1ex}-u^{\bar{q}p}\left[\mathscr{B}^{(p)}\overline{\ma}^{-1}\left( \mathscr{B}^{(q)}\right)^{*}\right]_{1\bar{1}} \\
		&\hspace{1ex}-\sum_{s\neq1}\frac{u^{\bar{q}p}}{\mu_s} \p \mm_{{1\bar{s}}}\q \mm_{s\bar{1}}-\sum_{s\neq1}\frac{u^{\bar{q}p}}{\mu_s}\q \mm_{{1\bar{s}}}\p \mm_{s\bar{1}}.
	\end{align}
	
First, we simplify (\ref{3.19}) and (\ref{3.20}).
		To facilitate readability, in Table \ref{ReadingTable} we put the components of (\ref{3.19}) and (\ref{3.20}) into four types.
	Specifying the indices $i=j=1$ for $\ma$ and $\mb$ in Section 2, we get 
	\begin{align}
		& u^{\bar{q}p}\left( \pq \ma - \q \mb \overline{\ma}^{-1}\p \overline{\mb}\right)_{1\bar{1}}= \underbrace{u^{\bar{q}p}u_{1s\bar{q}}(A^{s\bar{t}}-\ma^{s\bar{t}})u_{\bar{1}\bar{t}p} }_{\text{Type I}}\\
		&\underbrace{+hu^{\bar{q}p}(\mb_{1p}\mb_{\bar{1}\bar{q}}+\ma_{1\bar{q}}\ma_{p\bar{1}})}_{Type II} 
		\underbrace{-hu^{\bar{q}p}(u_{\bar{t}}u_{1s\bar{q}}\ma_{p\bar{1}} +u_su_{\bar{1}\bar{t}p}\ma_{1\bar{q}})\ma^{s\bar{t}}}_{Type III} \\
		& \underbrace{+h^2u^{\bar{q}p}(u_1u_{\bar{1}}\ma_{p\bar{q}}+u_1u_p\mb_{\bar{1}\bar{q}}+u_{\bar{1}}u_{\bar{q}}\mb_{1p}-u_su_{\bar{t}}\ma^{s\bar{t}}\ma_{1\bar{q}}\ma_{p\bar{1}})}_{ Type IV}.
	\end{align}
	We put these terms into the first row of the Table \ref{ReadingTable}. Then we have
	\begin{align}
		&u^{\bar{q}p}\left(\pq \overline{\ma}-\q \overline{\ma}\overline{\ma}^{-1}\p \overline{\ma}\right)_{\bar{1}1}=\underbrace{u^{\bar{q}p}u_{s\bar{1}\bar{q}}(A^{s\bar{t}}-\ma^{s\bar{t}}) u_{1\bar{t}p}}_{\text{Type I}} \\
		&\underbrace{+hu^{\bar{q}p}(\mb_{1p}\mb_{\bar{1}\bar{q}}+\ma_{1\bar{q}}\ma_{p\bar{1}})}_{Type II}
		\underbrace{-hu^{\bar{q}p}(u_{\bar{t}}u_{s\bar{1}\bar{q}}\mb_{1p}+u_su_{1\bar{t}p}\mb_{\bar{1}\bar{q}})\ma^{s\bar{t}}}_{ Type III}\\
		&  \underbrace{+h^2u^{\bar{q}p}(u_1u_{\bar{1}}\ma_{p\bar{q}}+u_pu_{\bar{1}}\ma_{1\bar{q}}+u_1u_{\bar{q}}\ma_{p\bar{1}}-u_su_{\bar{t}}\ma^{s\bar{t}}\mb_{1p}\mb_{\bar{1}\bar{q}})}_{Type IV}.
	\end{align}
	We put these terms into the second row of the Table \ref{ReadingTable}.
	Similarly, we get 
	\begin{align}
		&-u^{\bar{q}p}\left( \pq \mb- \q \mb \overline{\ma}^{-1}\p \overline{\ma}\right)_{11}=\underbrace{-u^{\bar{q}p}u_{1s\bar{q}}(A^{s\bar{t}}-\ma^{s\bar{t}})u_{1\bar{t}p}}_{Type I} \\
		&\underbrace{-hu^{\bar{q}p} ( \ma_{1\bar{q}}\mb_{1p}+\ma_{1\bar{q}}\mb_{1p})}_{Type II} \underbrace{+h u^{\bar{q}p}(u_{\bar{t}}u_{1s\bar{q}}\mb_{p1}+u_su_{1\bar{t}p}\ma_{1\bar{q}})\ma^{s\bar{t}}}_{Type III} \\
		&\underbrace{-h^2u^{\bar{q}p}(u_1u_1\ma_{p\bar{q}}+u_1u_p\ma_{1\bar{q}}+u_1u_{\bar{q}}\mb_{1p}-u_su_{\bar{t}}\ma^{s\bar{t}}\ma_{1\bar{q}}\mb_{1p})}_{Type IV}.
	\end{align}
	We put these terms into the third row of Table \ref{ReadingTable}.
	We also get 
	\begin{align}
		& -u^{\bar{q}p}\left( \pq \overline{\mb}-\q \overline{\ma}\overline{\ma}^{-1}\p\overline{\mb} \right)_{\bar{1}\bar{1}}=\underbrace{-u^{\bar{q}p}u_{\bar{1}s\bar{q}}(A^{s\bar{t}}-\ma^{s\bar{t}})u_{\bar{1}\bar{t}p}}_{Type I} \\
		&\underbrace{-hu^{\bar{q}p} (\ma_{p\bar{1}}\mb_{\bar{1}\bar{q}}+ \ma_{p\bar{1}}\mb_{\bar{q}\bar{1}})}_{Type II} \underbrace{+hu^{\bar{q}p}(u_{\bar{t}}u_{\bar{1}s\bar{q}}\ma_{p\bar{1}}+u_su_{\bar{1}\bar{t}p}\mb_{\bar{1}\bar{q}})\ma^{s\bar{t}}}_{Type III} \\
		&\underbrace{-h^2u^{\bar{q}p}( u_{\bar{1}}u_{\bar{1}}\ma_{p\bar{q}}+u_pu_{\bar{1}}\mb_{\bar{1}\bar{q}}+u_{\bar{q}}u_{\bar{1}}\ma_{p\bar{1}}-u_s u_{\bar{t}}\ma^{s\bar{t}}\mb_{\bar{1}\bar{q}}\ma_{p\bar{1}} )}_{Type IV}.
	\end{align}
	We put these terms into the fourth row of Table \ref{ReadingTable}.
	We then sum these components of $u^{\bar{q}p}\sigma_{p\bar{q}}$, place the sum in the last row of Table \ref{ReadingTable}.
	By direct computation, the sum of the Type II  contributions among the four terms in \eqref{3.19} and \eqref{3.20} is
	\begin{equation}
		\sim 2hu^{\bar{1}1}+ 2hu^{\bar{1}1}- 2hu^{\bar{1}1}- 2hu^{\bar{1}1}\sim 0,
	\end{equation}
	as clearly displayed in Table \ref{ReadingTable}. Similarly,  the sum of the Type III  contributions among the four terms in \eqref{3.19} and \eqref{3.20} is
	\begin{equation*}
		\sim 
		\left\{
		\begin{aligned}
			& -h\sum_t(u_{\bar{t}}u_{1t\bar{q}}u^{\bar{q}1}+u_tu_{\bar{1}\bar{t}p}u^{\bar{1}p}) -h\sum_t(u_{\bar{t}}u_{t\bar{1}\bar{q}}u^{\bar{q}1}+u_tu_{1\bar{t}p}u^{\bar{1}p})\\
			&+h\sum_t(u_{\bar{t}}u_{1t\bar{q}}u^{\bar{q}1}+u_tu_{1\bar{t}p}u^{\bar{1}p})+h\sum_t(u_{\bar{t}}u_{\bar{1}t\bar{q}}u^{\bar{q}1}+u_{t}u_{\bar{1}\bar{t}p}u^{\bar{1}p})\\
		\end{aligned}
		\right\}
		\sim 0.
	\end{equation*}
	The terms of Type I and Type IV should be treated together with other parts of $\Lsigma$.
	
	We then simplify the sum of the Type I terms, which is 
	\begin{equation}\label{3.38}
		u^{\bar{q}p}(u_{1s\bar{q}}-u_{\bar{1}s\bar{q}})(A^{s\bar{t}}-\ma^{s\bar{t}})(u_{\bar{1}\bar{t}p}-u_{1\bar{t}p}).
	\end{equation}
Denote 
	\begin{equation}
		W_{s\bar{q}}=u_{1s\bar{q}}-u_{\bar{1}s\bar{q}}.
	\end{equation}
	Then $W_{s\bar{q}}$ is a skew-Hermitian matrix, i.e., $ W_{s\bar{q}}=-\overline{W_{q\bar{s}}}$. With this notation, \eqref{3.38} becomes
	\begin{equation}
		u^{\bar{q}p}W_{s\bar{q}}(A^{s\bar{t}}-\ma^{s\bar{t}})\overline{W_{t\bar{p}}}.
	\end{equation}

		\begin{table}
			\centering
			\caption{4 types of components of \(u^{\bar{q}p}\sigma_{p\bar{q}}\) and their summation.}

			\renewcommand{\arraystretch}{1.20}
			\setlength{\tabcolsep}{6pt}
			\small 
			
			\begingroup
			\binoppenalty=10000
			\relpenalty=10000
			
			\begin{adjustbox}{max width=\linewidth,center}
				\begin{tabular}{@{}l l l l l@{}}
					\toprule	\label{ReadingTable}
					Components of $u^{\bar{q}p}\sigma_{p\bar{q}}$ & Type I & Type II & Type III & Type IV \\
					\midrule
					$u^{\bar{q}p}\left( \pq \ma - \q \mb \overline{\ma}^{-1}\p \overline{\mb}\right)_{1\bar{1}}$ & $u^{\bar{q}p}u_{1s\bar{q}}(A^{s\bar{t}}-\ma^{s\bar{t}})u_{\bar{1}\bar{t}p}$ & $\sim 2hu^{\bar{1}1}$ & $\sim -h\sum_t(u_{\bar{t}}u_{1t\bar{q}}u^{\bar{q}1}+u_tu_{\bar{1}\bar{t}p}u^{\bar{1}p})$ & \makecell[l]{$\sim h^2\sum_t (u^{\bar{t}t}u_1 u_{\bar{1}}+u^{\bar{1}t}u_1u_t$\\
						$\qquad\qquad+u^{\bar{t}1}u_{\bar{1}}u_{\bar{t}}-u^{\bar{1}1}u_{t}u_{\bar{t}})$}  \\
					\midrule
					$u^{\bar{q}p}\left(\pq \overline{\ma}-\q \overline{\ma}\overline{\ma}^{-1}\p \overline{\ma}\right)_{\bar{1}1}$ & $u^{\bar{q}p}u_{s\bar{1}\bar{q}}(A^{s\bar{t}}-\ma^{s\bar{t}}) u_{1\bar{t}p}$ & $\sim 2hu^{\bar{1}1}$ & $  \sim -h\sum_t(u_{\bar{t}}u_{t\bar{1}\bar{q}}u^{\bar{q}1}+u_tu_{1\bar{t}p}u^{\bar{1}p})$ & \makecell[l]{$\sim h^2\sum_t ( u^{\bar{t}t}u_1u_{\bar{1}}+u^{\bar{1}t}u_tu_{\bar{1}}$\\
						$\qquad\qquad+u^{\bar{t}1}u_1u_{\bar{t}}-u^{\bar{1}1}u_tu_{\bar{t}})$} \\
					\midrule
					$-u^{\bar{q}p}\left( \pq \mb- \q \mb \overline{\ma}^{-1}\p \overline{\ma}\right)_{11}$ & $-u^{\bar{q}p}u_{1s\bar{q}}(A^{s\bar{t}}-\ma^{s\bar{t}})u_{1\bar{t}p}$ & $\sim -2hu^{\bar{1}1}$ & $\sim h\sum_t(u_{\bar{t}}u_{1t\bar{q}}u^{\bar{q}1}+u_tu_{1\bar{t}p}u^{\bar{1}p})$ & \makecell[l]{$\sim-h^2\sum_t(u^{\bar{t}t}u_1u_1+u^{\bar{1}t}u_1u_t$\\
						$\qquad\qquad+u^{\bar{t}1}u_1u_{\bar{t}}-u^{\bar{1}1}u_tu_{\bar{t}})$} \\
					\midrule
					$-u^{\bar{q}p}\left( \pq \overline{\mb}-\q \overline{\ma}\overline{\ma}^{-1}\p\overline{\mb} \right)_{\bar{1}\bar{1}}$ & $-u^{\bar{q}p}u_{\bar{1}s\bar{q}}(A^{s\bar{t}}-\ma^{s\bar{t}})u_{\bar{1}\bar{t}p}$ & $\sim -2hu^{\bar{1}1}$& $\sim h\sum_t(u_{\bar{t}}u_{\bar{1}t\bar{q}}u^{\bar{q}1}+u_{t}u_{\bar{1}\bar{t}p}u^{\bar{1}p})$ & \makecell[l]{$\sim -h^2\sum_t( u^{\bar{t}t}u_{\bar{1}}u_{\bar{1}}+u^{\bar{1}t}u_tu_{\bar{1}}$\\
						$\qquad\qquad+u^{\bar{t}1}u_{\bar{t}}u_{\bar{1}}-u^{\bar{1}1}u_t u_{\bar{t}})$}  \\
					\midrule
					Sum of the above & $u^{\bar{q}p}W_{s\bar{q}}(A^{s\bar{t}}-\ma^{s\bar{t}})\overline{W_{t\bar{p}}}$ & $\sim 0$ & $\sim 0$ & $\sim h^2|u_1-u_{\bar{1}}|^2\sum_t u^{\bar{t}t}$\\
					\bottomrule
				\end{tabular}
			\end{adjustbox}
			
			\endgroup
		\end{table}

	To prove $u^{\bar{q}p}\pq \sigma \lesssim 0 $, it remains to show that
	\begin{align}
			&u^{\bar{q}p}W_{s\bar{q}}(A^{s\bar{t}}-\ma^{s\bar{t}})\overline{W_{t\bar{p}}}+h^2|u_1-\uonebar|^2\sum_su^{s\sbar} -\sum_{s\neq1}\frac{u^{\bar{q}p}}{\mu_s} \p \mm_{{1\bar{s}}}\q \mm_{s\bar{1}}\nonumber \\
			&-u^{\bar{q}p}\left[\mathscr{B}^{(p)}\overline{\ma}^{-1}\left( \mathscr{B}^{(q)}\right)^{*}\right]_{1\bar{1}} -\sum_{s\neq1}\frac{u^{\bar{q}p}}{\mu_s}\q \mm_{{1\bar{s}}}\p \mm_{s\bar{1}}	\lesssim 0.	\label{3.41}
		\end{align}

	In (\ref{naivepartialp_ijbar}) we let $i=1$ and $j=s$ and get
	\begin{align}
		\partial \MM_{1\sbar}=W_{p\sbar}+h(u_1-u_\onebar)\delta_{ps}-b_s\mathscr{B}^{( p)}_{ 1s},
	\end{align}
	which is a combination of $h(u_1-u_\onebar)$ and entries of $W$ and $\mathscr{B}^{( p)}_{ }$. It turns out that $\partial_p\MM_{s\onebar}$ is also a linear combinations of these terms. In (\ref{naivepartialp_ijbar}) if we directly replace $j$ by 1 and $i$ by $s$, we get
	\begin{align}
     \partial_p\MM_{s\onebar}=u_{sp\onebar}-b_su_{\sbar \onebar p}+hu_s\delta_{1p}-h b_su_{\sbar} \delta_{1p}-\mathscr{B}^{( p)}_{ s1}		.
	\end{align}
	It's not obvious that the above expression is a combination of $h(u_1-u_\onebar)$ and entries of $W$ and $\mathscr{B}^{( p)}_{ }$ but it's true if we switch from $\mathscr{B}^{( p)}_{ s1}$ to  $\mathscr{B}^{( p)}_{ 1s}$. First, directly differentiating $\MM=\MA-\MB\overline{\MA^{-1}}\ \overline{\MB}$ gives
	\begin{align}
\partial_p\MM_{s\onebar}=\partial_p\MA_{s\onebar}-\mathscr{B}^{( p)}_{ s1}-b_s\partial_p\overline{\MB_{s1}}.
\label{direct_diff_M}
	\end{align}
	We specify indices in (\ref{Notation_sb}) and get
	\begin{align}
		\mathscr{B}^{( p)}_{s1 }&=\partial_p\MB_{s1}-b_s\partial_p\MA_{1\sbar},\\
			\mathscr{B}^{( p)}_{1s}&=\partial_p\MB_{1s}-\partial_p\MA_{s\onebar}.
	\end{align}
	We notice that $\partial_p\MB_{s1}=\partial_p\MB_{1s}$, so
	\begin{align}
			\mathscr{B}^{( p)}_{s1 }=	\mathscr{B}^{( p)}_{1s}+\partial_p\MA_{s\onebar}-b_s\partial_p\MA_{1\sbar}.
	\end{align}
   Plugging this into (\ref{direct_diff_M}), we get
   \begin{align}
   	\partial_p\MM_{s\onebar}=&-\mathscr{B}^{(p )}_{ 1s}+b_s\left(\partial_p\MA_{1\sbar}-\partial_p\overline{\MB_{s1}}\right)\\
   	=&-\mathscr{B}^{(p )}_{ 1s}+b_sW_{p\sbar}+b_s\delta_{ps}h(u_1-u_\onebar).
   	\label{H_firstAppear}
   \end{align}

Now see the left-hand side of the inequality (\ref{3.41}) is a quadratic form of $W_{\pqbar}$, $\mathscr{B}^{(p)}_{1s}$ and $h(u_1-\uonebar)$. In the remaining part of this subsection, we reorganize it and in next subsection, we prove it's non-positive as a quadratic form only in the case of complex dimension 2.
Denote 
\begin{align}
	Q_1&=u^{\bar{q}p}W_{s\bar{q}}(A^{s\bar{t}}-\ma^{s\bar{t}})\overline{W_{t\bar{p}}},\\
	Q_2&=h^2|u_1-\uonebar|^2\sum_t u^{\bar{t}t},\\
	Q_3&=-\sum_{s\neq1}\frac{u^{\bar{q}p}}{\mu_s} \p \mm_{{1\bar{s}}}\q \mm_{s\bar{1}},\\
	Q_4&=-u^{\bar{q}p}\left[\mathscr{B}^{(p)}\overline{\ma}^{-1}\left( \mathscr{B}^{(q)}\right)^{*}\right]_{1\bar{1}} ,
\end{align}
and let 
$Q=Q_1+Q_2+Q_3+Q_4$. Then	\begin{align}
		u^{\bar{q}p}\sigma_{p\bar{q}}  \sim Q -\sum_{s\neq1}\frac{u^{\bar{q}p}}{\mu_s}\q \mm_{{1\bar{s}}}\p \mm_{s\bar{1}}.
	\end{align}
	In the following, we will work on each $Q_i$, making them easier to use. Then in section \ref{sec:2d} we prove that $Q\lesssim0$ in the case of complex dimension 2. We demonstrate the process using Figure \ref{fig:Process_Q_decompose}. The final result is in the dashed square.
		\begin{figure}[h]
		\centering
		\includegraphics[height=5.5cm]{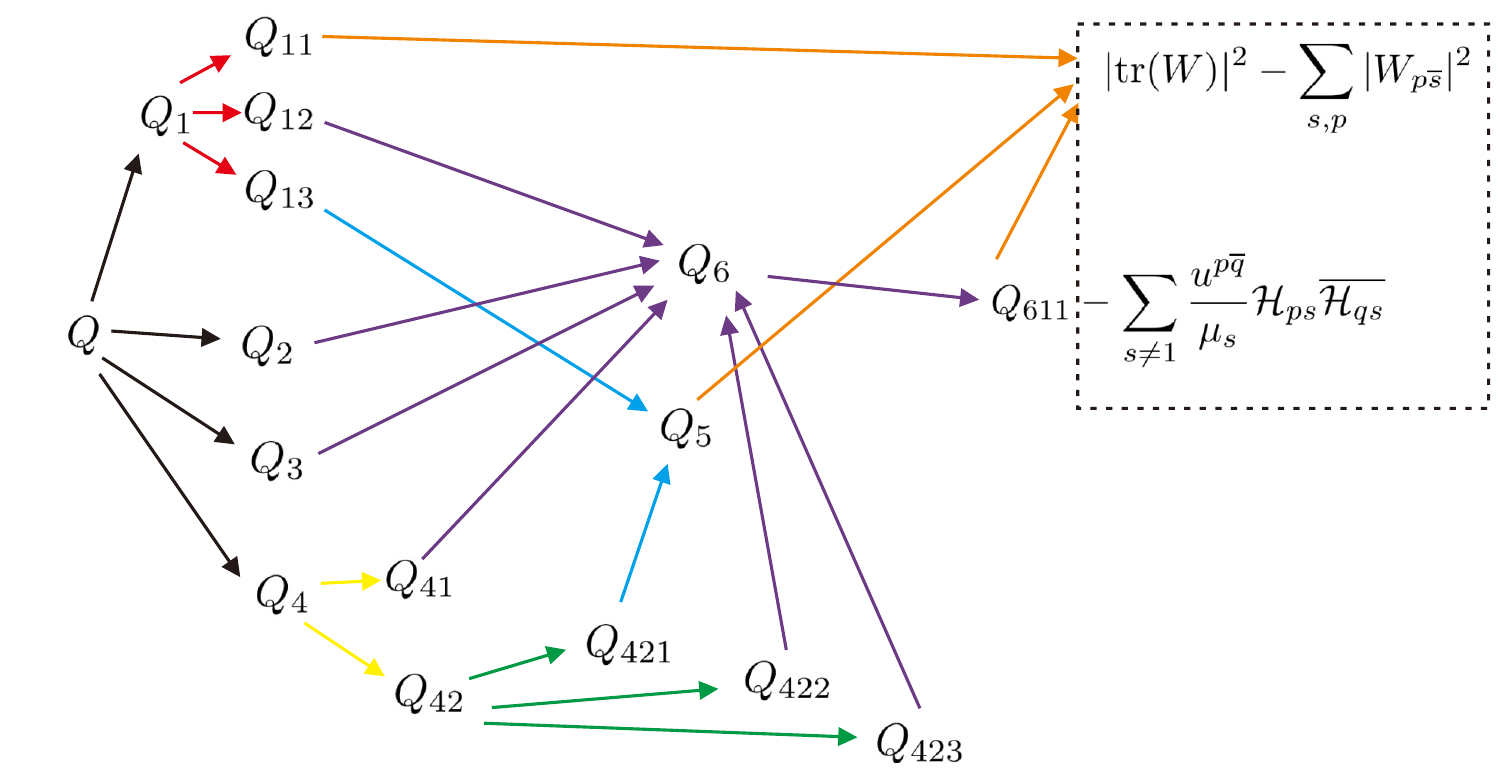}
		\caption{Transformation of $Q$}
		\label{fig:Process_Q_decompose}
	\end{figure}
	
	({\bf Step 1. Red Arrows in Figure \ref{fig:Process_Q_decompose}})
	To simplify $Q_1$, we need to use the relation between $A$ and $\MA$. Since
	\begin{equation}
		A= \ma - h(u_i u_{\bar{j}}), \text{  and } \ma=I,
	\end{equation}
	it follows that 
	\begin{align}
		A^{-1}=I+ \frac{h(u_iu_{\bar{j}})}{1-h\sum_k u_k u_{\bar{k}}}  \overset{\triangle}{=}  I+(r_ir_{\bar{j}}),
	\end{align}
	where $r_i= \sqrt{\frac{h}{1-h\sum_ku_ku_{\bar{k}}}}u_i$ and $r_{\bar{j}}=\overline{r_j}$ . With this structure, we get 
	\begin{equation}
		u^{\bar{q}p}=\delta^{pq}+r_p r_{\bar{q}} \;\text{ and }\; A^{s\bar{t}}-\ma^{s\bar{t}}=r_sr_{\bar{t}}.
	\end{equation}
	Then 
	\begin{align}
	Q_1
		&=\sum_{p,q,s,t}(\delta^{pq}+r_p r_{\bar{q}})W_{s\bar{q}}(r_sr_{\bar{t}})\overline{W_{t\bar{p}}} \\
		&=\sum_{p,q,s,t}\delta^{pq}W_{s\bar{q}}(r_sr_{\bar{t}})\overline{W_{t\bar{p}}}+\sum_{p,q,s,t}r_p r_{\bar{q}}W_{s\bar{q}}(r_sr_{\bar{t}})\overline{W_{t\bar{p}}}.\label{3.12.1}
	\end{align}
	Using $W_{p\bar{s}}u^{\bar{s}p}=0$, we can simplify the second term in the above identity,
	\begin{align}
		\sum_{p,q,s,t} r_p r_{\bar{q}}W_{s\bar{q}}(r_sr_{\bar{t}})\overline{W_{t\bar{p}}}&=  \sum_{p,q,s,t} r_s r_{\bar{q}}W_{s\bar{q}}(r_pr_{\bar{t}})\overline{W_{t\bar{p}}}\\
		&= \sum_{p,q,s,t}  - (u^{\bar{q}s}-\delta^{qs})W_{s\bar{q}}(u^{\bar{t}p}-\delta^{tp}) W_{p\bar{t}}\\
		&= - \delta^{qs}W_{s\bar{q}}\delta^{tp}W_{p\bar{t}}\\
		&=   \delta^{qs}W_{s\bar{q}} \delta^{tp} \overline{W_{t\bar{p}}}=   \bigl|\mathrm{tr}\bigl( W_{t\bar{q}}\bigr)\bigr|^2 \triangleq Q_{11}   .    \label{3.12.2}
	\end{align}
	Relabeling the summation indices in the first term of \eqref{3.12.1} yields
	\begin{equation}\label{3.12.3}
		\sum_{s,t,p,q}\delta^{pq}W_{s\bar{q}}(r_sr_{\bar{t}})\overline{W_{t\bar{p}}}=  \sum_{t\neq 1,p,q}r_pr_{\bar{q}}W_{p\bar{t}}\overline{W_{q\bar{t}}}+\sum_{p,q}r_pr_{\bar{q}}W_{p\bar{1}}\overline{W_{q\bar{1}}}.
	\end{equation}
	We denote the two terms on the right-hand side of the above equation by $Q_{12}$ and $Q_{13}$ respectively. We lable them differently since $Q_{13}$ cancels with other terms.

		({\bf Step 2. Yellow Arrows in Figure \ref{fig:Process_Q_decompose}})
		In $Q_4$, terms with $s\neq 1$ should be treated differently from those with $s=1$. 
		\begin{align}
			Q_4&=-\sum_{p,q,s}u^{\pqbar} \mathscr{B}^{(p)}_{1s}\overline{\mathscr{B}^{(q)}_{1s}}\\
			&=-\sum_{s\neq1, p,q}u^{\pqbar} \mathscr{B}^{(p)}_{1s}\overline{\mathscr{B}^{(q)}_{1s}}-\sum_{p,q}u^{\pqbar} \mathscr{B}^{(p)}_{11}\overline{\mathscr{B}^{(q)}_{11}}.
		\end{align}
		Denote the last two terms by $Q_{41}$ and $Q_{42}$ respectively.

			({\bf Step 3. Green Arrows in Figure \ref{fig:Process_Q_decompose}})
			Term $Q_{42}$ can be further simplified.
		    Since $\partial_p\sigma\sim0$ and at the point of computation we assume $\MK^1_1=\MM_{1\qbar}\MA^{1\qbar}$, we have
		    \begin{align}
		    	\partial_p\MM_{1\onebar}\sim 0.
		    \end{align}
		    More explicitely, this is
		    \begin{align}
		    	W_{p\onebar}+h(u_1-u_{\onebar})\delta_{p1}- \mathscr{B}^{(p)}_{11}\sim0.
		    						\label{relation_WB}
		    \end{align}
			This provides a relation between $	W_{p\onebar}+h(u_1-u_{\onebar})\delta_{p1}$ and $\mathscr{B}^{(p)}_{11}$. Plugging this into $Q_{42}$, we get
			\begin{align}
				Q_{42}&\sim-u^{\pqbar}(W_{p\onebar}+h(u_1-u_\onebar)\delta_{p1})\overline{(W_{q\onebar}+h(u_1-u_\onebar)\delta_{q1})}\\
				&=-u^{\pqbar}W_{p\onebar}\overline{W_{q\onebar}}-u^{1\qbar}h(u_1-u_\onebar)\overline{W_{q\onebar}}-u^{p\onebar}W_{p\onebar}\overline{h(u_1-u_{\onebar})}\nonumber\\
				&\ \ \ \ \ \ \ \ \ \ \ \ \ \ \ \ \ \ \ \ \ \ \ \ \ \ \ \ \ \ \ \ \ \ \ \ \ \ \ \ \ \ \ \ \ \ \ \ \ \ \  \ \ \ \ \ \ \ \ \ -u^{1\onebar}|h(u_1-u_\onebar)|^2\\
				&\triangleq Q_{421}+Q_{422}+Q_{423},
			\end{align}
			 where 
			 \begin{align}
			 	Q_{421}&=-u^{\pqbar}W_{p\onebar}\overline{W_{q\onebar}},\\
			 	 Q_{422}&=-u^{1\qbar}h(u_1-u_\onebar)\overline{W_{q\onebar}}-u^{p\onebar}W_{p\onebar}\overline{h(u_1-u_{\onebar})},\label{firstExpression_Q422}\\
			 	 Q_{423}&=-u^{1\onebar}|h(u_1-u_\onebar)|^2.
			 \end{align}
			 We need to transform $Q_{422}$ using the relation $u^{\pqbar}W_{\pqbar}=0$, which says
			 \begin{align}
			 	\sum_{s\neq 1}u^{p\sbar}W_{p\sbar}\sim-u^{p\onebar}W_{p\onebar}.
			 \end{align}
			 Plugging this into (\ref{firstExpression_Q422}), we get
			 \begin{align}
			 	Q_{422}=\sum_{s\neq 1}u^{s\qbar}h(u_1-u_\onebar)\overline{W_{q\sbar}}
			 	             +  \sum_{s\neq 1}u^{p\sbar}W_{p\sbar}\overline{h(u_1-u_{\onebar})}.
			 \end{align}
			
				({\bf Step 4. Blue Arrows in Figure \ref{fig:Process_Q_decompose}})
				By putting $Q_{13}$ and $Q_{421}$ together we can make them simplier:
				\begin{align}
					Q_{13}+Q_{421}=r^p\rqbar W_{p\onebar}\overline{W_{q\onebar}}-u^{\pqbar}W_{p\onebar}\overline{W_{q\onebar}}=-\sum_p|W_{p\onebar}|^2.
				\end{align}

					({\bf Step 5. Purple Arrows in Figure \ref{fig:Process_Q_decompose}})
					
					Denote that 
					\begin{align}
						Q_6=Q_{12}+Q_2+Q_3+Q_{41}+Q_{422}+Q_{423}.
						\label{Q6Sum}
					\end{align}
					In the summands on the right-hand side of (\ref{Q6Sum}), we collect all the terms which are quadratic forms of $W$ and denote their summation by $Q_{61}$:
					\begin{align}
						Q_{61}=\sum_{s\neq 1}W_{p\sbar}\overline{W_{q\sbar}}\left(r^p\rqbar-\frac{u^{\pqbar}}{\mu_s}\right).
					\end{align}
						Using $\mu_s=1-b_s^2$, the above expression can be simplified
					\begin{equation}
						r_pr_{\bar{q}}-\frac{u^{\bar{q}p}} {\mu_s}  =  r_pr_{\bar{q}}-u^{\bar{q}p}-\frac{u^{\bar{q}p}b_s^2}{\mu_s}=-\delta^{pq}-\frac{u^{\bar{q}p}b_s^2}{\mu_s}.
					\end{equation}
				Then we denote $Q_{61}=Q_{611}+Q_{612}.$ Where
				\begin{align}
					Q_{611}&=-\sum_{s\neq 1, p}W_{p\sbar}\overline{W_{p\sbar}},\\
				   	Q_{612}&=-\sum_{s\neq 1, p,q}\frac{u^{p\qbar}b_s^2}{\mu_s}W_{p\sbar}\overline{W_{q\sbar}}.
				\end{align}
				In (\ref{Q6Sum}), we put the terms containing both $W$ and $\mathscr{B}^{(p)}$ together and denote their sum by $Q_{62}$:
				\begin{align}
					Q_{62}=\sum_{s\neq 1, p, q} W_{p\sbar} \overline{  \mathscr{B}^{(q)}_{1s} } b_s\frac{u^{\pqbar}}{\mu_s}+\overline{W_{q\sbar} }{  \mathscr{B}^{(p)}_{1s} } b_s\frac{u^{\pqbar}}{\mu_s}
				\end{align}
					In (\ref{Q6Sum}), we put the terms which are quadratic forms of $\mathscr{B}^{(p)}$  together and denote their summation by $Q_{63}$:
					\begin{align}
						Q_{63}=-\sum_{s\neq 1, p, q}u^{\pqbar} \mathscr{B}^{(p)}_{1s}\overline{\mathscr{B}^{(q)}_{1s}}
						-\sum_{s\neq 1, p, q}u^{\pqbar} \frac{b_s^2}{\mu_s}\mathscr{B}^{(p)}_{1s}\overline{\mathscr{B}^{(q)}_{1s}}.
					\end{align}On the right-hand side of the equation above, two summands are contributed by $Q_{41}$ and $Q_3$ respectively.
					 Using $\mu_s=1-b_s^2$, $Q_{63}$ can be simplified:
					 \begin{align}
					 	Q_{63}=	-\sum_{s\neq 1, p, q} \frac{u^{\pqbar}}{\mu_s}\mathscr{B}^{(p)}_{1s}\overline{\mathscr{B}^{(q)}_{1s}}.
					 \end{align}
					 
					In (\ref{Q6Sum}), we put the terms which contain both $h(u_1-u_\onebar)$ and $W$ together and denote their summation by $Q_{64}$:
					\begin{align}
						Q_{64}=&\left(\sum_{s\neq 1, q}u^{s\qbar}\overline{W_{q\sbar}}{h(u_1-u_\onebar)}+\sum_{s\neq 1, q}u^{p\sbar}{W_{p\sbar}}\overline{h(u_1-u_{\onebar})}\right)\nonumber\\
						&\ -\left(\sum_{s\neq 1, q}\frac{u^{s\qbar}}{\mu_s}\overline{W_{q\sbar}}{h(u_1-u_\onebar)}+\sum_{s\neq 1, q}\frac{u^{p\sbar}}{\mu_s}{W_{p\sbar}}\overline{h(u_1-u_{\onebar})}\right).
					\end{align}
					Terms in the two brackets in the equation above are contributed by $Q_{422}$ and $Q_3$ respectively. 	 Using $\mu_s=1-b_s^2$, $Q_{64}$ can be simplified:
							\begin{align}
						Q_{64}=-\left(\sum_{s\neq 1, q}\frac{u^{s\qbar}b_s^2}{\mu_s}\overline{W_{q\sbar}}{h(u_1-u_\onebar)}+\sum_{s\neq 1, q}\frac{u^{p\sbar}b_s^2}{\mu_s}{W_{p\sbar}}\overline{h(u_1-u_{\onebar})}\right).
					\end{align}

						In (\ref{Q6Sum}), we put the terms which contain both $h(u_1-u_\onebar)$ and $\mathscr{B}^{( p)}$ together and denote their summation by $Q_{65}$:
								\begin{align}
						Q_{65}=-\left(\sum_{s\neq 1, q}\frac{u^{s\qbar}b_s}{\mu_s}\overline{\mathscr{B}^{(q )}_{ 1s}}{h(u_1-u_\onebar)}
						   +\sum_{s\neq 1, q}\frac{u^{p\sbar}b_s}{\mu_s}{\mathscr{B}^{( p)}_{1s }}\overline{h(u_1-u_{\onebar})}\right).
					\end{align}
							In (\ref{Q6Sum}), we put the terms which contain only $h(u_1-u_\onebar)$ together and denote their summation by $Q_{66}$:
							\begin{align}
								Q_{66}=|h(u_1-u_\onebar)|^2\sum_su^{s\sbar}
								-|h(u_1-u_\onebar)|^2\sum_{s\neq 1}\frac{u^{s\sbar}}{\mu_s}
								-|h(u_1-u_\onebar)|^2u^{1\onebar}.
								\label{Q66initial}
							\end{align}
							Three terms on the right-hand side of (\ref{Q66initial}) are contributed by $Q_2, Q_3, Q_{23}$ respectively. After cancellation and simplification, we get
							\begin{align}
								Q_{66}=
								-|h(u_1-u_\onebar)|^2\sum_{s\neq 1}\frac{u^{s\sbar}b_s^2}{\mu_s}.
							\end{align}
							By choosing
							\begin{align}
								H_{ps}=W_{p\sbar} b_s-\mathscr{B}^{( p)}_{ 1s}+h(u_1-u_\onebar)\delta_{sp}b_s,
								\label{H_second}
							\end{align} we find $Q_6-Q_{611}$ can be largely simplifed as
							\begin{align}
								Q_6-Q_{611}=-\sum_{s\neq 1, p, q} \frac{u^{p\qbar}}{\mu_s}H_{ps}\overline{H_{qs}}.
							\end{align}
						Comparing (\ref{H_firstAppear}) and (\ref{H_second}), we find $H_{ps}$ is exactly $\partial_p{\MM_{s\onebar}}$.
							

						({\bf Step 6. Orange Arrows in Figure \ref{fig:Process_Q_decompose}})
						
						\begin{align}
							Q_{11}+Q_5+Q_{611}&=|\trace(W)|^2-\sum_p|W_{p\onebar}|^2-\sum_{s\neq 1, p}|W_{p\overline{s}}|^2\\
							&=|\trace(W)|^2-\sum_{s, p}|W_{p\overline{s}}|^2.
						\end{align}


Finally, we get
	\begin{align}
		Q&\sim\bigl|\mathrm{tr}\bigl( W\bigr)\bigr|^2 -\sum_{p,s}|W_{p\bar{s}}|^2- \sum_{s\neq 1,q,p}\frac{u^{\bar{q}p}}{\mu_s} H_{ps} \overline{H_{qs}},
	\end{align}
	which is in the dashed square in Figure \ref{fig:Process_Q_decompose}.
  In addition, we have
  \begin{align}
  	u^{p\qbar	}\partial_{\pqbar}\sigma\sim\bigl|\mathrm{tr}\bigl( W\bigr)\bigr|^2 -\sum_{p,s}|W_{p\bar{s}}|^2- 2\sum_{s\neq 1,q,p}\frac{u^{\bar{q}p}}{\mu_s} H_{ps} \overline{H_{qs}}.
  	\label{final_fullQ}
  \end{align}

	\subsection{Proving Non-Positivity in the Case of Complex Dimension 2}
	\label{sec:2d}
	So we want to prove that for any anti-Hermitian matrix $(W_{p\bar{s}})$,
	\begin{align}
		\sum_{p,s} W_{p\bar{s}}\overline{W_{p\bar{s}}} -\bigl|\mathrm{tr}\bigl( W_{p\bar{s}}\bigr)\bigr|^2 \ge 0,\label{8.1}
	\end{align}
	with the restriction  
	\begin{align}
		\sum_{p,q} W_{p\bar{q}}u^{\bar{q}p}=\sum_{p,q} W_{p\bar{q}}(\delta^{\bar{q}p} +r_p r_{\bar{q}})=0.\label{8.2}
	\end{align}
	It turns out that this is valid only in complex dimension two. When the complex dimension is greater than two, the quadratic form (modulo \eqref{8.2}) may not be positive definite.
	In the following, we prove this in the case of complex dimension two. Denote 
	\begin{equation}
		(W_{p\bar{s}})=
		\begin{pmatrix}
			\sqrt{-1}a & \sqrt{-1}b+d \\
			\sqrt{-1}b-d & \sqrt{-1}c
		\end{pmatrix},
	\end{equation}
	where $a,b,c,d$ are real numbers.
	Let $r_1= w_1+\sqrt{-1}v_1$ and $r_2= w_2+\sqrt{-1}v_2$.
	Then 
	\begin{align*}
		&(u^{\bar{q}p})=
		\begin{pmatrix}
			1+r_1\overline{r_{1}} & r_1\overline{r_{2}} \\
			\overline{r_{2}}r_1 & 1+r_2\overline{r_{2}}
		\end{pmatrix}\\
		&=
		\begin{pmatrix}
			1+w_1^2+v_1^2 & w_1w_2+v_1v_2+\sqrt{-1}(v_1w_2-w_1v_2)\\
			w_1w_2+v_1v_2-\sqrt{-1}(v_1w_2-w_1v_2) & 1+w_2^2+ v_2^2
		\end{pmatrix}.
	\end{align*}
	Relation \eqref{8.1} is equivalent to 
	\begin{equation}
		a^2 +c^2 +2b^2+2d^2 \ge (a+c)^2.
	\end{equation}
	Expanding the right-hand side and cancelling $a^2$ and $c^2$, we obtain
	\begin{equation}
		b^2+d^2\ge ac.\label{9.2}
	\end{equation}
	Relation \eqref{8.2} is equivalent to 
	\begin{equation}
		a(1+w_1^2+v_1^2)+c(1+w_2^2+ v_2^2)+2b(w_1w_2+v_1v_2)-2d(v_1w_2-w_1v_2)=0.
	\end{equation}
	Using this relation in \eqref{9.2}, we replace $a$ by a combination of $b,c,d$, and transform \eqref{9.2} into 
	\begin{equation*}
		(b^2+d^2)(1+w_1^2+v_1^2)\ge -c^2(1+w_2^2+ v_2^2)-2bc(w_1w_2+v_1v_2)+2dc(v_1w_2-w_1v_2).
	\end{equation*}
	It is equivalent to the positive semi-definiteness of the following matrix
	\begin{equation}
		\begin{pmatrix}\label{9.3}
			1+w_2^2+ v_2^2& w_1w_2+v_1v_2 & w_1v_2-v_1w_2 \\
			w_1w_2+v_1v_2 & 1+w_1^2+ v_1^2 & 0\\
			w_1v_2-v_1w_2 &0 &  1+w_1^2+v_1^2
		\end{pmatrix}.
	\end{equation}
	According to Lemma \ref{A2}, it suffices to show that the following quantity is nonnegative:
	\begin{equation}
		(1+w_1^2+v_1^2)(1+w_2^2+ v_2^2)- (w_1w_2+v_1v_2)^2- (v_1w_2-w_1v_2 )^2\ge 0.
	\end{equation}
	With basic computation, the above inequality equals to 
	\begin{equation}
		1+w_1^2+v_1^2 +w_2^2+ v_2^2 >0.
	\end{equation}
	So the matrix in \eqref{9.3} is positive definite and we have proved that \eqref{8.1} is nonnegative. 
	\begin{remark}
		The condition \eqref{8.1} may not be valid in general dimension. For example, in the case of dimension n, let 
		\begin{equation}
			W_{p\bar{s}}=
			\begin{cases}
				\sqrt{-1}\delta_{ps},& p>1, \\
				-(n-1)\frac{\sqrt{-1}}{C^2+1}\delta_{ps}, &p=1.
			\end{cases}
		\end{equation}
		Let $r_1=C$ , $r_k=0$ for $k>1$, where $C$ is a real constant. Then $(W_{p\bar{s}})$ is skew-Hermitian and the condition \eqref{8.2} is satisfied. Under this setting, we obtain 
		\begin{equation}
			\mathrm{tr} (W_{p\bar{s}})=(n-1)\sqrt{-1}(1-\frac{1}{C^2+1}). 
		\end{equation}
		So the condition \eqref{8.1} becomes 
		\begin{equation}\label{11.1}
			(n-1)+\frac{(n-1)^2}{(C^2+1)^2} -\left[(n-1)\left( 1-\frac{1}{C^2+1}\right) \right]^2\ge 0.
		\end{equation}
		This is equivalent to 
		\begin{equation}\label{11.2}
			n-1-(n-1)^2+ \frac{2(n-1)^2}{C^2+1}   = n-1+(n-1)^2\left(\frac{2}{C^2+1}-1\right) \ge 0.
		\end{equation}
		It is only valid for $n=1$ and $n=2$. When $n>1$, the inequality \eqref{11.2} is equivalent to 
		\begin{equation}
			1 \ge (n-1)\left(1-\frac{2}{C^2+1}\right) .
		\end{equation}
		For any fixed $n>2$, we can choose $C$ large enough, such that 
		\begin{equation}
			1 < (n-1)\left(1-\frac{2}{C^2+1}\right).
		\end{equation}
		
	\end{remark}

	\begin{remark}
		After showing that 
		$|\trace(W)|^2-\sum_{p,s} W_{p\bar{s}}\overline{W_{p\bar{s}}} $ may not be non-positive in general dimesion, the readers may still wander if this is true for $u^\pqbar\sigma_{p\bar{q}}$. Since
		\begin{align}
			u^\pqbar\sigma_{p\bar{q}}\sim|\trace(W)|^2-\sum_{p,s} W_{p\bar{s}}\overline{W_{p\bar{s}}} 
			-2\sum_{s\neq 1,q,p}\frac{u^{\bar{q}p}}{\mu_s} H_{ps} \overline{H_{qs}},
		\end{align} it's reasonable to ask if the subtraction of $-2\sum_{s\neq 1,q,p}\frac{u^{\bar{q}p}}{\mu_s} H_{ps} \overline{H_{qs}}$ makes a difference. The answer is negative. In the following situation the contribution of $-2\sum_{s\neq 1,q,p}\frac{u^{\bar{q}p}}{\mu_s} H_{ps} \overline{H_{qs}}$ is zero. 
		
	Let $b_s=0$ for $s\neq 1$, which is a possible situation.
	In this situation 
	\begin{align}
			u^\pqbar\sigma_{p\bar{q}}\sim|\trace(W)|^2-\sum_{p,s} W_{p\bar{s}}\overline{W_{p\bar{s}}} 
		-2\sum_{s\neq 1,q,p}\frac{u^{\bar{q}p}}{\mu_s} \mathscr{B}^{( p)}_{1s}
		 \overline{ \mathscr{B}^{( q)}_{1s} }.
	\end{align}
	$W$ is the same as that of  the previous remark.
	We notice that $W$ is diagonal.  The only relation between $\mathscr{B}^{( p)}_{ }$ and $W$ is (\ref{relation_WB}), which put a restriction on $W_{p\onebar}$ and $\mathscr{B}^{( p)}_{11 }$. Also we notice that $\mathscr{B}^{(p)}_{1s }=\mathscr{B}^{(s )}_{1 p}$, so 
	\begin{align}
	-2\sum_{s\neq 1,q,p}\frac{u^{\bar{q}p}}{\mu_s} \mathscr{B}^{( p)}_{1s}
	\overline{ \mathscr{B}^{( q)}_{1s} }=	-2\sum_{s\neq 1,q,p}\frac{u^{\bar{q}p}}{\mu_s} \mathscr{B}^{( s)}_{1p}
		\overline{ \mathscr{B}^{( s)}_{1q} },
	\end{align}
	and terms related to $W$ are
	\begin{align}
			-2\sum_{s\neq 1}\frac{u^{\bar{1}1}}{\mu_s} \mathscr{B}^{( s)}_{11}
		\overline{ \mathscr{B}^{( s)}_{11} }=-2\sum_{s\neq 1}|W_{1s}|^2.
	\end{align}
	By our choice, all the off-diagonal elements of $W$ are $0$, so the contribution above is $0$.
	\end{remark}
	So, we get at the place where the  minimum eigenvalue of $\MK<$ the second minimum eigenvalue of $\MK$
	\begin{align}
	u^{\ijbar}\partial_{\ijbar}\sigma\leq C_0(\sigma+|\nabla \sigma|).
	\label{sigma_super_harmonic}
	\end{align}
	
	This result implies the following proposition, which is crucial in the lower bound estimate for $\sigma$:
	\begin{proposition}
		\label{prop:touch}
	At the place where the  minimum eigenvalue of $\MK<$ the second minimum eigenvalue of $\MK$, $\sigma$ cannot achieve a $0$ local  minimum, unless $\sigma\equiv 0$.
	\end{proposition}
	\begin{proof}
		When the  minimum eigenvalue of $\MK<$ the second minimum eigenvalue of $\MK$, $\sigma$ is a smooth function and it satisfies (\ref{sigma_super_harmonic}). The result follows from Harnak inequality.
	\end{proof}

	\subsection{The General Case via Approximation}
	\label{sec:generalCase_comparison}

	Proposition \ref{prop:touch} in the previous section cannot directly provide a lower bound estimate for $\sigma$. As a function defined in $\Omega$, $\sigma$ is in general only $C^0$ and at the place where $\sigma$ achieves a local minimum,  the  minimum eigenvalue of $\MK$ may equal to the second minimum eigenvalue of $\MK$. In this section, we perturb $\sigma$ by quadratic polynomials and get $\sigma^\epsilon$. We will show that when  $\sigma^\epsilon$ achieves a local minimum at a point $z_0$, we do have that  the  minimum eigenvalue of $\MK^q(z_0)<$ the second minimum eigenvalue of $\MK^q(z_0)$, where $\MKq$ is computed using $u+q$ (parallel to that $\MK$ is computed using $u$) where $q$ is a pluriharmonic quadratic polynomial with $L^2$ norm restriction.

   Since $\Omega$ is a strictly convex domain with smooth boundary and $u$ is $C^2$ continuous with nonvanishing gradient near $\partial\Omega$, $u$ has strictly convex level sets near  $\partial\Omega$. So for a small enough $\varepsilon$, with
   \begin{align}
   	\Ov=\left\{
   	z\in\Omega| \ \text{dist}(z, \partial\Omega)>\varepsilon
   	\right\}
   \end{align} $-\sqrt{-u}$ is strictly convex on $\partial\Ov$. 
   In the following, we show that  $-\sqrt{-u}$ is stricly convex in $\Ov$, and we need to first derive an apriori estimate. 
   
      Let \begin{align}
   	\sQe=\left\{
   	q \left|\ \text{quadratic polynomial, } q_{\ijbar}=0,\  \int_{B_1(\ba)} q^2\leq\right. \epsilon
   	\right\}.
   \end{align}
   In above, $B_1(\ba)$ is a ball of radius $1$ with respect to the standard metric on $\EC^n$, with center $\ba\in \Omega$; the integration also uses the standard measure on $\EC^n$.
   \begin{proposition}
   	\label{prop:perturbation_stable}
   	Suppose that $u<0$ and $-\sqrt{-u}$ is strictly convex in $\overline{\Ov}$. $\epsilon$ is small enough so that for any $q\in\sQe$, 
   	\begin{align}
   &	\ \ \ \ \ \ 	u+q<0, &\text{ \ \ \ \ \ in  \   } \overline{\Ov}, \label{ass:negative}\\
   	&		u+q \text{ is strictly pluri-subharmonic}, &\text{ \ \ \ \ \ in    \  } \overline{\Ov},\label{ass:plurisubharmonic}
   	\end{align} and 
   	\begin{align}
   		-\sqrt{-u-q} \text{ is strictly convex on } \partial{\Ov}.\label{ass:strongly_convex_boundary}
   	\end{align}
   	Then for any $q\in\sQe$
   		\begin{align}
   		-\sqrt{-u-q} \text{ is strictly convex in }  {\Ov}.\label{ass:strongly_convex_inside}
   	\end{align}
   	In addition, we can find a constant $C$ depending on $\epsilon$, $\min_{\overline{\Ov}} u$,  $\max_{\overline\Ov}u$ and the diameter of $\Omega$, so that
   	\begin{align}
   		 \text{the real Hessian of }-\sqrt{-u}>CI_{2n}.
   	\end{align} 
   \end{proposition}

   \begin{proof}
   	[Proof of the Proposition \ref{prop:perturbation_stable}]

   Given $q\in \sQe$, let
		\begin{align}
	\MAq=\left[(u+q)_\ijbar+\frac{(u+q)_i(u+q)_\jbar}{2(-u-q)}\right];\\
	\MBq=\left[(u+q)_{ij}+\frac{(u+q)_i(u+q)_j}{2(-u-q)}\right],
\end{align}
 and $ \MMq=\MAq-\MBq(\overline{\MAq})^{-1}\overline{\MBq},\ 
 \MKq=\MMq(\MAq)^{-1}.$ In above that $\MAq$ is invertible is guaranteed by the assumption (\ref{ass:plurisubharmonic}). Denote
 \begin{align}
 	\lamb(\MKq)(z_0)&= \text{ the minimum eigenvalue of } \MKq(z_0),\\
 	\samb(\MKq)(z_0)&= \text{ the second minimum eigenvalue of } \MKq(z_0),\\
 	\sigma^\epsilon(z_0)&=\min_{q\in\sQe}\lamb(\MKq(z_0)).
 \end{align}
Since $\sQ^\deltaa\subset\sQ^\deltab$ when $\deltaa\leq \deltab$, $\sigma^\deltaa\geq \sigma^\deltab$. So $\sigma^s$ depends on $s$ monotonically. Since $\lamb(\MKq)$ depends on $q$ algebraically and $\MA^q$ is invertible for $q\in\sQe$, we have that at any fixed point $z_0$, $\sigma^s(z_0)$ is a continuous function of $s$ for $s\leq \epsilon$.
\begin{figure}[h]
	\centering
	\includegraphics[height=4.5cm]{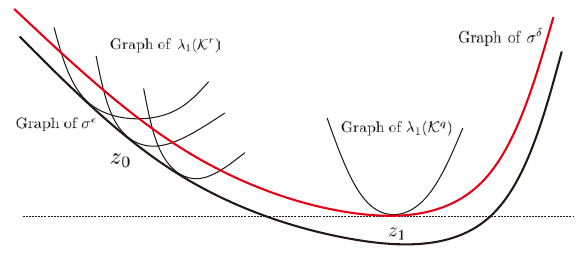}
	\caption{ Graph of  $\sigma^\epsilon$ is the Lower Envelop of  the Graph of  $\lambda_1(\MKq)$     with $q\in \sQe$}
	\label{fig:Envelop}
\end{figure}

At any point $z_0\in\overline{\Ov}$, since $\sQe$ is closed and of finite dimension, we can find $r\in \sQe$ so that
\begin{align}
	\lambda_1(\MK^r)(z_0)=\sigma^\epsilon(z_0),
\end{align}
and for $z\neq z_0$
\begin{align}
	\lambda_1(\MK^r)(z)\geq \sigma^\epsilon(z).
\end{align}
This means the graph of $\lambda_1(\MK^r)$ contacts with the graph of $ \sigma^\epsilon$ from above. 
In addition, Lemma \ref{lemma:Seperation} implies 
\begin{align}
		\lambda_1(\MK^r)(z_0)<	\lambda_2(\MK^r)(z_0).
\end{align} So  $\lambda_1(\MK^r)$ is a smooth function around $z_0$. Also since $u+r$ is still a solution to the Monge-Amp\`ere equation, $\det\left((u+r)_{\pqbar}\right)=1$, $\lamb(\MK^r)$ also satisfy (\ref{sigma_super_harmonic})  with $C_0$ depending on $u+r$. 
\begin{remark}
	\label{viscosity}
	The graph of $\sigmae$ is the lower envelop of a family of locally defined smooth functions ($\lamb (\MK^r)$ defined around $z_0$) satisfying (\ref{sigma_super_harmonic}) with $C_0$ depending on $u+r$. 
But we cannot say that $\sigma^\epsilon$ satisfies \begin{align}
	u^{\ijbar}\partial_{\ijbar}\sigmae\leq C_1(\sigmae+|\nabla \sigmae|)
\end{align}in the sense of viscosity
because it's not obvious why we can find a uniform $C_1$ depending on $u$ and $\epsilon$.

\end{remark}

  The assumption (\ref{ass:strongly_convex_boundary}) implies that $\sigmae>0$ on $\partial\Ov$. If $\sigmae$ is not positive in $\Ov$, then for a $\delta\leq \epsilon$, $\sigma^\delta$ achieves a $0$ local minimum  at a point $z_1\in\Ov$. Then for a $q\in \sQ^\delta$, $\lambda_1(\MKq)$ achieves a $0$ local minimum at $z_1$. 
    Since $\lambda_1(\MKq)$ cannot achieve  $0$ interior local minimum according to Proposition \ref{prop:touch}, we get a contradiction.
So we proved that $\sigmae>0$ in $\Ov$. It follows that for any $q\in \sQe$, $\lamb(\MKq)>0$ and $-\sqrt{-u-q}$ is strictly convex. Using Lemma \ref{lemma:robustness}, we have a convexity estimate for $-\sqrt{-u}$. This proves the proposition.
   \end{proof}
   
  For a smooth and strictly convex domain $\Omega$, we connect it with a ball. Let 
  \begin{align}
  	\Omega_t=t\Omega+(1-t)B,
  \end{align}where $B$ is a ball of radius 1. Let $u_t$ be the solution to the following Dirichlet problem:
  \begin{align}
  	u_t=0,\ \ \ \ \ \text{on }\partial\Omega_t,
  \end{align}
  and 
  \begin{align}
  	det(\partial_{\ijbar}u_t)=1,\text{\ \ \ \ \ \ in }\Omega_t.
  \end{align} 
  In particular, we have $u_0=|z|^2-1$ and $u_1=u$.
  We can find a small $\varepsilon>0$, so that $-\sqrt{-u_t}$ is strictly convex on $\partial\Ov_t$ for all $t\in[0,1]$. 
  For each $t$, we can compute $\sigmae_t$ using $u_t$, parallel to the previous computation, replacing $u$ by $u_t$. Then  for a fixed $\epsilon$ small enough $\sigmae_t>0$ on $\partial\Ov_t$ and $u_t+q$ is strictly plurisubharmonic for $q\in\sQe$ and for all $t\in[0,1]$. 
 
  Let $m(t)=\text{min}_{\overline{\Ov_t}}\sigmae_t$. Then $m(t)$ is a continuous function of $t$, with $m(0)>0$. If $m(1)>0$, $-\sqrt{-u}$ is strictly convex and the result is proved. Otherwize, let $s_0=\min \left\{
    s| m(s)=0
   \right\}.$ For $s<s_0$, $\sigma_s^\epsilon>0$. So for $q\in\sQe$, $u_s+q$ is strictly convex. Lemma \ref{lemma:robustness} implies a convexity estimate for $-\sqrt{-u_s}$, with $s<s_0$, depending on $\epsilon$ and the diameter of $\Omega$. Taking limit, we know $-\sqrt{-u_{s_0}}$ is strictly convex. This allow us to apply Proposition \ref{prop:perturbation_stable} to $u_{s_0}$. So $\sigma_{s_0}^\epsilon>0$ in $\Ov_{s_0}$, which is a contradiction.
   
   Therefore $m(1)>0$, $-\sqrt{-u}$ is strictly convex and the convexity estimate follows from Lemma \ref{lemma:robustness}.

	\appendix
	\section{Algebra Lemmas}\label{app1}

	\begin{lemma}
		\label{lem:determinant_lemma}
		Suppose that $u$ is a quadratic polynomial on $\mathbb{C}^n$. Let $A=(u_{i\bar{j}})$ and $B=(u_{ij})$. Here $u_{i\bar{j}}$ and $u_{ij}$ are derivatives with respect to the complex coordinates $z^i$. Let $D$ be the real Hessian matrix of $u$:
		\begin{equation}
			\begin{pmatrix}
				D_{2i-1,2j-1} & D_{2i-1,2j} \\
				D_{2i,2j-1} & D_{2i,2j}
			\end{pmatrix}
			=
			\begin{pmatrix}
				\partial_{x^i}\partial_{x^j}u & \partial_{x^i}\partial_{y^j}u \\
				\partial_{y^i}\partial_{x^j}u & \partial_{y^i}\partial_{y^j}u
			\end{pmatrix},
		\end{equation}
		where $x^i+\sqrt{-1}y^i=z^i$. Then
		\begin{equation}\label{A1-1}
			\det\left(I-B\overline{A}^{-1}\overline{B}A^{-1}\right)=\frac{\det D}{4^n\left(\det A\right)^2}.
		\end{equation}
	\end{lemma}
	\begin{proof}
		Firstly, we show that the identity above is invariant under a linear change of coordinates. Then we choose complex coordinates so that $A=I$ and $B=\mathrm{diag}(b_1,b_2,\cdots, b_n)$, with $b_k \in \mathbb{R}$, and we prove the identity (\ref{A1-1}) in this case.    
		
		Let  $\{\widetilde{z^i}=z^jK^i_j\}$ be another set of coordinates. Then
		$\widetilde{A}=(\partial_{\widetilde{z^i}}\overline{\partial_{\widetilde{z^j}}} u)$, $\widetilde{B}=(\partial_{\widetilde{z^i}}\partial_{\widetilde{z^j}} u)$, and $\widetilde{D}$ be the real Hessian matrix of $u$
		with coordinates $(\widetilde{x^i},\widetilde{y^i})$ with $\widetilde{x^i}+\sqrt{-1}\widetilde{y^i}=\widetilde{z^i}$.
		We have $A=K \widetilde{A}K^*$ and $B=K\widetilde{B}K^T$. Then 
		\begin{equation}
			B\overline{A}^{-1}\overline{B}A^{-1}=K \left( \widetilde{B} \overline{\widetilde{A}}^{-1}\overline{\widetilde{B}}\widetilde{A} \right)K^{-1}.
		\end{equation}
		When the coordinates are changed, the matrix $ B\overline{A}^{-1}\overline{B}A^{-1}$ transforms by a similarity transformation. 
		Consequently, the matrices $ B\overline{A}^{-1}\overline{B}A^{-1}$ and $\widetilde{B} \overline{\widetilde{A}}^{-1}\overline{\widetilde{B}}\widetilde{A}$ have the same eigenvalues.
		Therefore, the left-hand side of (\ref{A1-1}) is invariant under a linear change of coordinates.
		
		Set $K=P+\sqrt{-1}Q$. Under this notation, the matrices $D$ and $\widetilde{D}$ satisfy the following relation: 
		\begin{equation}
			D=R\widetilde{D}R^T,
		\end{equation}
		with 
		\begin{equation}
			\begin{pmatrix}
				R^{2i-1,2j-1} & R^{2i-1,2j} \\
				R^{2i,2j-1} & R^{2i,2j}
			\end{pmatrix}
			=
			\begin{pmatrix}
				P^i_j & -Q^i_j \\
				Q^i_j & P^i_j
			\end{pmatrix}.
		\end{equation}
		With the new coordinate $\widetilde{z^i}=\widetilde{x^i}+\sqrt{-1}\widetilde{y^i}$, the right-hand side of (\ref{A1-1}) becomes
		\begin{equation}
			\frac{\det \widetilde{D}}{4^n (\det \widetilde{A})^2} =\frac{ \det D |\det K|^4}{4^n (\det A)^2 (\det R)^2}.
		\end{equation}
		We need to show $|\det K|^4=(\det R)^2$. This is a basic linear algebraic fact and its proof can be found in 
		\cite{Horn-Johnson2013}, 1.3, p21 (d). Therefore, the right-hand side of (\ref{A1-1}) is invariant under a linear change of coordinate.  
		
		We first choose a coordinate such that $A=I$. Then we choose a unitary coordinate transformation so that 
		$B=\mathrm{diag}(b_1,b_2,\cdots, b_n)$, with $b_k \in \mathbb{R}$. This is possible according to Autonne-Takagi Factorization ( Lemma A.3 of \cite{hu25}) or Corollary 4.4.4(c) of \cite{Horn-Johnson2013}.
		
		In this case, we have 
		\begin{equation}
			I-B\overline{A}^{-1}\overline{B}A^{-1}=\mathrm{diag}(1-b_1^2, 1-b_2^2, \cdots, 1-b_n^2),
		\end{equation}
		and 
		\begin{equation}
			D=\mathrm{diag}\left(2(1+b_1),2(1-b_1),\cdots,2(1+b_n),2(1-b_n) \right).
		\end{equation}
		It follows that 
		\begin{equation}
			\det (  I-B\overline{A}^{-1}\overline{B}A^{-1}) =\Pi_{k=1}^n(1-b_k^2)=\frac{\det D}{4^n}=\frac{\det D}{4^n (\det A)^2}.
		\end{equation}
		This proves (\ref{A1-1}).
	\end{proof}
	
	\begin{lemma}\label{A2}
		Let constants $B,C>0$ and  $A,\beta, \gamma \in \mathbb{R}$. The matrix 
		\begin{equation}
			M=
			\begin{pmatrix}
				A & \beta &  \gamma \\
				\beta & B & 0 \\
				\gamma &0 & C
			\end{pmatrix}
		\end{equation}
		is positive semi-definite if and only if 
		\begin{equation}
			A -\frac{\beta^2}{B}-\frac{\gamma^2}{C}\ge 0.
		\end{equation}
		\end{lemma}
		\begin{proof}
			Let the matrix 
			\begin{equation}
				P=
				\begin{pmatrix}
					1 & 0 & 0\\
					-\frac{\beta}{B} & 1 & 0 \\
					-\frac{\gamma}{C}  &0 & 1
				\end{pmatrix} ,
			\end{equation}
			which is invertible. It follows that
			\begin{equation}
				P^TMP=
				\begin{pmatrix}
					A -\frac{\beta^2}{B}-\frac{\gamma^2}{C} & 0 & 0 \\
					0 &B & 0\\
					0&0 & C
				\end{pmatrix}.
			\end{equation}
			Since congruence by an invertible matrix preserves positive semi-definiteness, we obtain
			\begin{equation}
				M \ge 0 ,\quad \text{ if and only if } \quad  P^TMP\ge 0.
			\end{equation}
			Moreover, since $B,C> 0$, the matrix $P^TMP\ge 0$ is equivalent to
			\begin{equation}
				A -\frac{\beta^2}{B}-\frac{\gamma^2}{C} \ge 0.
			\end{equation}
		\end{proof}
		
		The following lemma is used when doing the perturbation in section \ref{sec:generalCase_comparison}.  It says the perturbation which minimizes the minimal eigenvalue of $\MKq$ must make the minimal eigenvalue of $\MKq$ strictly smaller than then second minimal eigenvalue of $\MKq$.
		\begin{lemma}
			\label{lemma:Seperation}
			$u$ is a $C^2$, negative function with $-\sqrt{-u}$ being strictly convex, defined in a neighborhood of $z_0\in \EC^n$. $\sQe$ is a set of pluriharmonic quadratic polynomials on $\EC^n$ with $L^2$ norm restriction:
			\begin{align}
				\sQe=\left\{
				q \big| \text{real-valued quadratic polynomial, } q_{\ijbar}=0,\ \int_{\bK}q^2\leq \epsilon
				\right\}.
			\end{align}
			In above  $\bK$ is a bounded and closed set in $\EC^n$ with smooth boundary; the integration uses a measure that is equivalent to the standard metric on $\EC^n$. Let 
			\begin{align}
				\MAq=\left[(u+q)_\ijbar+\frac{(u+q)_i(u+q)_\jbar}{2(-u-q)}\right](z_0);\\
					\MBq=\left[(u+q)_{ij}+\frac{(u+q)_i(u+q)_j}{2(-u-q)}\right](z_0).
			\end{align}
			Then we combine them to get $\MMq$ and $\MKq$:
		   \begin{align}
		   	 \MMq&=\MAq-\MBq(\overline{\MAq})^{-1}\overline{\MBq};\\
		   	 \MKq&=\MMq(\MAq)^{-1}.
		   \end{align}
		   If $\epsilon$ is small enough so that for any $q\in \sQe$
		   \begin{align}
		   u+q<0, \ \ \ \text{\ \ at }z_0
		   	\label{assumption:small_epsilon_negative}
	   \end{align}
   and \begin{align}
   	-\sqrt{-u-q} \text{ is convex  and strictly pluri-subharmonic at $z_0$},
   	\label{assumption:small_epsilon_convex}
   \end{align}
	  then the $q\in \sQe$ which minimizes the minimal eigenvalue of $\MKq$ makes 
		   \begin{align}
		   	  \mineig  (\MKq) <\smineig(\MKq).
		   \end{align}
		\end{lemma}
		\begin{proof}
			In the proof of this lemma, we denote the minimum eigenvalue and the second minimal eigenvalue of a matrix $H$ by $\lamb(H)$ and $\samb(H)$ respectively. Also, in the integration, we ignore the  region $\bK$ since we only do integration in this region.
			
			We prove this lemma by contradiction. Suppose that $q\in \sQe$ minimizes  $\lamb(\MKq)$, i.e.
			\begin{align}
				\lamb(\MKq)=\min_{r\in \sQe} \lamb(\MK^r),
			    \label{assumption:q_minimal}
			\end{align}
			 and 
			\begin{align}
				\lamb(\MKq) =\samb(\MKq).  \label{assumption:equal_min_second_min_eig}
			\end{align}
			Then we derive a contradiction. Note that for each fixed $\epsilon$, $\sQe$ is a closed set of finite dimension and $\lambda_1(\MKq)$ is a continuous function of $q$ so a $q$ satisfying (\ref{assumption:q_minimal}) exists.
			
			To simplify the computation, we change coordinate on $\EC^n$ linearly so that
			\begin{align}
				\MAq&=I_n;\\
				\MBq&=\diag(b_1, b_2,\ ...\ , b_n),
			\end{align}
			where $1\geq b_1\geq b_2\geq\ ...\ \geq b_n\geq 0$.
			This is possible because as we change coordinate linearly on $\EC^n$, matrices $\MAq$ and $\MBq$ change in the following way:
		\[\MAq\rightarrow P\MAq P^\ast, \ \ \ \MBq\rightarrow P\MBq P^T, \] where $P$ is the coordinate transition matrix. Then using Autonne-Takagi factorization, we can diagonalize $\MAq$ and $\MBq$ simultaneously, where $\MBq$ becomes diagonal with non-negative diagonal elements
		 (see Lemma A.3 in \cite{hu25} or Corollary 4.4.4(c) in \cite{Horn-Johnson2013}). 
	 Then 
	 \begin{align}
	 	\MKq=\diag(1-b_1^2,\ 1-b_2^2,\ ...\ , 1-b_n^2).
	 \end{align}	
	 Using Lemma \ref{lem:determinant_lemma}, we know that  the assumption (\ref{assumption:small_epsilon_convex}) implies $1-b_i^2\geq 0$. So all $b_i$'s are $\leq 1$.
	Because of the assumption (\ref{assumption:equal_min_second_min_eig}), we have $b_1=b_2$.
		
		Choose  quadratic polynomials $h, g$, so that
		\begin{align}
			h_{\ijbar}=0, \ \ (h_{ij})=\diag(1,-1,0,\ ...\ ,0),\ \ h(z_0)=0,\ \ \nabla h(z_0)=0;\\
			g_{\ijbar}=0, \ \ \ \ (g_{ij})=\diag(1,0,\ ...\ ,0),\ \ \ \ \ g(z_0)=0,\ \ \nabla g(z_0)=0.
		\end{align}
		The following are the only four cases:
				\begin{itemize}
			\item Case 1, $b_1=0$. Let  $w(t)=tg$.
			\item Case 2, $b_1>0$, $\int qh\neq 0$.  Let  $w(t)=q+th$.
			\item Case 3, $b_1>0$, $\int qh= 0$, $q\not\equiv 0$. \\ \ \ Let  $w(t)=(1-\gamma t^2)q+th$, for large enough $\gamma$.
			\item Case 4, $b_1>0$, $q\equiv 0$. Let  $w(t)=tg.$
		\end{itemize}
		In each case, we can always find a $t_0$ small enough so that  $\MK^{w(t_0)}$ has a smaller minimum eigenvalue,
		\begin{align}
			\lamb(\MK^{w(t_0)})<\lamb(\MK^{q}),
		\end{align}
		 and $w(t_0)\in \sQe$.
		
		({\bf Case 1},  $b_1=0$) Since $b_1= b_2\geq \ ...\ b_n\geq 0$, we have $b_i=0$ for all $i=1,\ ...\ , n$. Then by the construction of  $\MKq$, all eigenvalues of $\MKq$ are $1$. $q$ is chosen to make 
		\begin{align}
			\lamb(\MKq)=\min_{r\in \sQe} \lamb(\MK^{r})
			\label{eq:Case1_minimal}
		\end{align}
		so 
			\begin{align}
			\lamb(\MKq)\leq  \lamb(\MK^{0}).
		\end{align}
		Therefore all eigenvalue of $\MK^0$ must be $1$ and naturally $\MK^0=I_n$. Since 
		\begin{align}
			\MK^0=I_n-\MB^0(\overline{\MA^0})^{-1}\overline{\MB^0}({\MA^0})^{-1},
		\end{align} we have $\MB^0=0$. So 
		\begin{align}
			\MB^{tg}=\diag(t,\ 0,\ ...\ ,0)\neq0.
		\end{align} Then $\MK^{tg}\neq I_n$. Since $\MK^{tg}\leq I_n$, we have 
		\begin{align}
			\lamb(\MK^{tg})<1.
		\end{align} When $t$ is very small $tg\in \sQe$, so we find an element $tg$ in $\sQe$ which makes $\MK^{tg}<\MKq$. This contradicts with the assumption (\ref{assumption:q_minimal}).
		
			({\bf Case 2},  $b_1>0$, $\int qh\neq 0$) 
		 Let $w(t)=q+th$.  We will show that for some $t>0$, $w(t)\in\sQe$ and $\lamb(\MK^{w(t)})<\lamb(\MKq)$, which is a contradiction.
		 
		 Since $\nabla h(z_0)=0$ and $h(z_0)=0$, $\MA^{w(t)}$ and $\MB^{w(t)}$ have very simple expressions:
		 \begin{align}
		 	\MA^{w(t)}&=\MAq=I_n\ \ \ \ \ \ (\text{since } h_{\ijbar}=0);\\
		 	\MB^{w(t)}&=\MBq+t \diag(1,-1,0,\ ...\ ,0)=\diag(b_1+t, b_2-t, b_3,\ ...\ , b_n).
		 \end{align}  
		  So 
		  \begin{align}
		  		\MK^{w(t)}&=\diag(1-b_1^2-2b_1t-t^2, 1-b_2^2+2b_2t-t^2, 1-b_3^2,\ ...\ , 1-b_n^2).
		  \end{align}
		Because $b_1=b_2\geq b_3\geq \ ...\ \geq b_n\geq 0$, when $t$ is very small, 
		\begin{align}
			\lamb(\MK^{w(t)})=\left\{
			\begin{array}{cc}
					1-b_1^2-2b_1t-t^2< 1-b_1^2=\lamb(\MKq) & \text{if \ \ \ } t>0;\\
			1-b_1^2+2b_1t-t^2< 1-b_1^2=\lamb(\MKq) & \text{if \ \ \ } t<0.
			\end{array}
				\right.
		\end{align}
 So when $t$ is small enough and $t\neq 0$ 
 \begin{align}
 	\lamb(\MK^{w(t)})<\lamb(\MK^q).
 \end{align}
	We also need to show that $w(t)$ stays in $\sQe$, which only requires that 
		\begin{align}
			\int w(t)^2\leq \epsilon,
		\end{align}
		because  $h, \ q$ are both pluriharmonic quadratic polynomials.
		Direct computation gives
		\begin{align}
		\int w(t)^2=\int (q+th)^2	=\int q^2+2t\int hq+t^2\int h^2.
		\end{align}
		Case 2 requires that  $\int hq\neq 0$ so 
		\begin{align}
			\left.\frac{d}{dt}\right|_{t=0} \int w(t)^2=2\int hq\neq 0.
		\end{align}
		Therefore, for a very small $t$ ($t>0$, if $\int hq<0$; $t<0$, if $\int hq>0$) 
		\begin{align}
			\int w(t)^2<\int w(0)^2=\int q^2\leq \epsilon.
		\end{align}
		
		({\bf{Case 3, $b_1>0,\ \int qh=0, \ q\not\equiv 0$}})  
		Let $w(t)=(1-\gamma t^2)q+th$. We will show that for a very large $\gamma$, and for $t$ small enough (depending on $\gamma$), 
		\begin{align}
			\MK^{w(t)}&< \MKq, \ \ \ {t\neq 0},\label{case3_K}\\
			w(t)&\in \sQe,\label{case3_norm}
		\end{align}which is a contradiction.
		
		$\MA^{q+th}$ and $\MB^{q+th}$ are the same as those of  Case 2:
				 \begin{align}
			\MA^{q+th}&=I_n;\\
			\MB^{q+th}&=\diag(b_1+t, b_2-t, b_3,\ ...\ , b_n).
		\end{align}  
		Adding $-\gamma t^2 q$ to $q+th$ introduces a perturbation of order $t^2$ to the above matrices, so
	    \begin{align}
			\MA^{w(t)}&=I_n+O(t^2);\\
			\MB^{w(t)}&=\diag(b_1+t, b_2-t, b_3,\ ...\ , b_n)+O(t^2).
		\end{align}  In above $O(t^2)$ are matrices depending on $\gamma$, which is to be determined.
		Then
		\begin{align}
			\MK^{w(t)}=\diag(1-b_1^2-2b_1t, 1-b_1^2+2b_1t, 1-b_3^2, \ ...\ , 1-b_n^2)+O(t^2).
		\end{align}So  when $t$ is small enough
		\begin{align}
	\lamb(\MK^{w(t)})=\left\{
	\begin{array}{cc}
		1-b_1^2-2b_1t+O(t^2)< 1-b_1^2=\lamb(\MKq) & \text{if \ } t>0;\\
		1-b_1^2+2b_1t+O(t^2)< 1-b_1^2=\lamb(\MKq) & \text{if \ } t<0.
	\end{array}
	\right.
\end{align}	(\ref{case3_K}) is verified. Then we need to show $w(t)\in\sQe$, which requires $\int w(t)^2\leq \epsilon$. Direct algebraic computation gives
\begin{align}
	\int w(t)^2=\int q^2+ 2t\int hq+t^2\int(h^2-2\gamma q^2)+  O(t^3).
\end{align}
		Case 3 requires that $\int hq=0$, so 
		\begin{align}
			\int w(t)^2=\int q^2+t^2\int(h^2-2\gamma q^2)+  O(t^3).
		\end{align}
		As a function of $t$, it achieves a local maximum at $t=0$ providing that $\gamma$ is large enough so that $\int h^2-2\gamma q^2<0$.  (\ref{case3_norm}) is verified.

		({\bf Case 4, $b_1>0$, $q\equiv 0$}) Let  $w(t)=tg.$ Since $q=0$ and $g$ is pluri-harmonic with
		$g(z_0)=0$, $\nabla g(z_0)=0$,  we have 
		\begin{align}
		   \MA^{w(t)}&=I_n;\\
		   	\MB^{w(t)}&=\diag(b_1+t, b_2,\ ...\ , b_n ).
		\end{align}
		Then
		\begin{align}
			\MK^{w(t)}= \diag(1-b_1^2-2b_1t-t^2, 1-b_2^2,\ ...\ , 1-b_n^2).
		\end{align}
		We choose $t$ small enough and $t>0$, then
		\begin{align}
				\lamb(\MK^{w(t)})=
			1-b_1^2-2b_1t-t^2< 1-b_1^2=\lamb(\MKq).
		\end{align}
	When $t$ is small enough, it's obvious that $\int (tg)^2\leq\epsilon$, so $w(t)=tg\in \sQe$.
		
					\end{proof}
				
				The following lemma says that if under the perturbation of many $q$, $-\sqrt{-u-q}$ are all convex, then the real Hessian of $-\sqrt{-u}$ has a positive lower bound. This is an analog of  Lemma A.8 of \cite{hu25} and Lemma A.6 of \cite{hu-maxrank-2024}.
				
				\begin{lemma}
					\label{lemma:robustness}
					Suppose that $u$ is defined in a neighborhood of $p_0\in \EC^n$. Let
					\begin{align}
						\sQe=\left\{
						q\left| \text{  real-valued quadratic polynomial, }  q_{\ijbar}=0,\ \int_{\bK}q^2\leq \epsilon\right.
						\right\}.
					\end{align}
						In above  $\bK$ is a bounded and closed set in $\EC^n$ containing $0$ with smooth boundary; the integration uses a measure that is equivalent to the standard metric on $\EC^n$. 
					If for any $q\in\sQe$, $u+q<0$ and $-\sqrt{-u-q}$ is striclty convex at $z_0$, then there is a constant $C>0$ depending on $u(z_0), \epsilon$, $|z_0|$, $\bK$ and the measure used in the integration so that
					\begin{align}
						\text{the real Hessian of } -\sqrt{-u}  \text{ at }z_0\geq CI_{2n}. 
					\end{align}
					
				\end{lemma}
\begin{proof} In the following, we only do estimate at $z_0$ so when there is no ambiguity, we omit $z_0$.
	Let 
		\begin{align}
		\MAq=\left[(u+q)_\ijbar+\frac{(u+q)_i(u+q)_\jbar}{2(-u-q)}\right](z_0);\\
		\MBq=\left[(u+q)_{ij}+\frac{(u+q)_i(u+q)_j}{2(-u-q)}\right](z_0).
	\end{align}
	Denote $\MA=\MA^0$ and $\MB=\MB^0$. Then 
	\begin{align}
		\left(
		\partial_{\ijbar} (-\sqrt{-u})
		\right)(z_0)=\frac{ \MA}{2\sqrt{-u}(z_0)},\\
			\left(
		\partial_{ij} (-\sqrt{-u})
		\right)(z_0)=\frac{ \MB}{2\sqrt{-u}(z_0)}.
	\end{align}
	If we can find a constant $C_1$, so that
	\begin{align}
	&	\ \ \ \ \ \ \  \MA> C_1 I_n; \label{cond:Alowerbound}\\
	&	\MB\overline{(\MA-C_1I_n)^{-1}}\ \overline{\MB}(\MA-C_1I_n)^{-1}< 1\label{cond:BABA},
	\end{align}
	using Lemma A.6 of \cite{hu25}, we can prove that 
	\begin{align}
		-\sqrt{-u}-\frac{C_1}{2\sqrt{-u(z_0)}}|z|^2\text{ is strictly convex at $z_0$},
	\end{align}providing a desired estimate. In the following we derive
	(\ref{cond:Alowerbound}) and (\ref{cond:BABA}) using Lemma A.8 of \cite{hu25}.
	
	When $C_2$ is large enough, depending on $|z_0|$, let $\delta=\frac{\epsilon}{C_2}$, then $\sQe$ contains a subset $\sRd$:
	\begin{align}
		\sRd=\left\{
	\Ree(B_{ij}(z-z_0)_i(z-z_0)_j) \left| B \text{ symmetric,}\ B \overline{B}<\delta\right.
		\right\}.
	\end{align}
	For $q\in \sRd$,
	\begin{align}
		\MA^q&=\MA;\\
		\MB^q&=\MB+(q_{ij}).
	\end{align} Since $-\sqrt{-u-q}$ is strictly convex, for $q\in \sRd$, we have $\MAq>0$ and
	\begin{align}
		\MB^q\overline{(\MA^q)^{-1}}\ \overline{\MB^q}(\MA^q)^{-1}<1. 
	\end{align} This says for any symmetric complex matrix $B$ with $ B \overline{B}<\delta$, 
	\begin{align}
			(\MB+B)\overline{\MA^{-1}}\ \overline{(\MB+B)}\MA^{-1}<1.
	\end{align}
	Then the estimate (\ref{cond:Alowerbound}) and (\ref{cond:BABA}) follows from Lemma A.8 of \cite{hu25}.
\end{proof}
	
	\section{Derivatives of the Minimal and Maximal Eigenvalues of a Matrix Valued Function}
	Let $K=K^i_j$ be a smooth section of $T_{1,0}\otimes T^*_{1,0}$ defined in a neighborhood of a point ${\bf{z}} \in \mathbb{C}^n$. Let $X$, $Y$ be two constant vector fields on $\mathbb{C}^n$. For the first and second derivatives of the minimum and the maximum eigenvalues of $K$, we have the following theorem:
	\begin{lemma}
		\label{lem:derivates_min_eig}
		Let $\lambda_{max}(K)$ (resp. $\lambda_{min}(K)$) be the maximal (resp. minimal) eigenvalue of $K$.
		Suppose that at ${{\bf{z}}}$, $K=\mathrm{diag}(k_1,k_2,\cdots, k_n)$. 
		If $k_1>k_2\ge \cdots \ge k_n$, then at ${\bf{z}}$, 
		\begin{align}
			\partial_X  \lambda_{max}(K)&= \partial_X K_1^1, \label{B1-1}\\
			\partial_X\partial_{Y} \lambda_{max}(K)&=\partial_{XY}K^1_1 +\sum_{\theta\neq1} \frac{\partial_X K_1^\theta \partial_{Y} K_{\theta}^1}{k_1-k_{\theta}}+\sum_{\theta\neq 1}\frac{\partial_{Y} K_1^\theta \partial_{X} K_{\theta}^1}{k_1-k_{\theta}}. \label{B1-2}
		\end{align}
		If $k_1< k_2 \le \cdots \le k_n$, then at ${\bf{z}}$, 
		\begin{align}
			\partial_X \lambda_{min}(K) &= \partial_X K_1^1, \label{B2-1}\\
			\partial_X\partial_{Y} \lambda_{min}(K)&=\partial_{XY}K^1_1 +\sum_{\theta\neq1} \frac{\partial_X K_1^\theta \partial_{Y} K_{\theta}^1}{k_1-k_{\theta}}+\sum_{\theta\neq 1}\frac{\partial_{Y} K_1^\theta \partial_{X} K_{\theta}^1}{k_1-k_{\theta}}. \label{B2-2}   
		\end{align}
	\end{lemma}
	
	\begin{proof}
		First, we note that we only need to prove the theorem in the case that all the eigenvalues of $K$ are positive. Let 
		$\widetilde{K}=K+v \cdot Id$, where $v$ is a constant. Then the eigenvalues of $\widetilde{K}$ and $K$ differ only by the constant $v$, so their derivatives are the same. When $v$ is large enough, all eigenvalues of $\widetilde{K}$ are positive.
		We prove (\ref{B1-1}) and (\ref{B1-2}) using an approximation method. Then using the relation
		\begin{equation}
			\text{ Minimal Eigenvalue of $K$} = \frac{1 }{\text{Maximal Eigenvalue of $K^{-1}$}},
		\end{equation}
		we can prove (\ref{B2-1}) and (\ref{B2-2}).
		
		We use 
		\begin{align}
			H^{(p)} = (\mathrm{tr} K^p)^{\frac{1}{p}}, \quad p \in \mathbb{Z}^{+},
		\end{align}
		to approximate the maximal eigenvalue of $K$. It's a basic calculus fact that 
		\begin{align}
			\lim_{p\rightarrow \infty} \left( \sum_{j=1}^N a_j^p \right)^{\frac{1}{p}}=\max \{ a_1,\cdots, a_N \},
		\end{align}
		with $a_i>0$. So 
		\begin{equation}
			\lim_{p\rightarrow \infty} H^{(p)} = \lambda_{max}(K).
		\end{equation}
		We will show 
		\begin{align}
			&(1) \lim_{p\rightarrow \infty} \partial_{X} H^{(p)}= \partial_X (K^1_1), \label{B3-1}\\
			&(2) \lim_{p\rightarrow \infty} \partial_{XY}H^{(p)}=\partial_{XY}K^1_1 +\sum_{\theta\neq1} \frac{\partial_X K_1^\theta \partial_{Y} K_{\theta}^1}{k_1-k_{\theta}}+\sum_{\theta\neq 1}\frac{\partial_{Y} K_1^\theta \partial_{X} K_{\theta}^1}{k_1-k_{\theta}}, \label{B3-2}\\
			&(3) \text{The right-hand side of (\ref{B3-2}) is $\partial_X\partial_{Y} \lambda_{max}(K)$}.\label{B3-3}   
		\end{align}
		Then (\ref{B1-1}) and (\ref{B1-2}) are proved. 
		
		At ${\bf{z}}$, combining the matrix $K=\mathrm{diag}(k_1,k_2,\cdots, k_n)$ and  $k_1>k_2\ge \cdots \ge k_n$ yields 
		\begin{align*}
			\lim_{p\rightarrow \infty}&\frac{ K^{p-1}}{\left(\mathrm{tr}K^{p}\right)^{\frac{p-1}{p}}}=\lim_{p\rightarrow \infty} \frac{1}{k_1^{p-1}}\mathrm{diag}(k_1^{p-1},k_2^{p-1}, \cdots, k_n^{p-1})=
			\mathrm{diag}(1,0, \cdots, 0).
		\end{align*}
		It follows that 
		\begin{equation}\label{B3-4}
			(1)\, \lim_{p\rightarrow \infty}\partial_X H^{(p)}=\lim_{p\rightarrow \infty} \frac{\mathrm{tr}\left( K^{p-1}\partial_XK \right)}{\left(\mathrm{tr}K^p\right)^{\frac{p-1}{p}}}=\partial_X (K^1_1).
		\end{equation}
		So, (\ref{B3-1}) is proved.
		
		Next, we compute the second derivative of $H^{(p)} $,
		\begin{align}
			(2)    \partial_{XY}H^{(p)} &= \frac{\mathrm{tr}\left( K^{p-1}\partial_{XY}K \right) }{\left(\mathrm{tr}K^p\right)^{\frac{p-1}{p}}}+ \sum_{j=0}^{p-2}\frac{\mathrm{tr}\left( K^j \partial_{Y}K K^{p-2-j}\partial_X K \right)}{\left(\mathrm{tr}K^p\right)^{\frac{p-1}{p}}} \nonumber\\
			&\hspace{2ex}-(p-1) \frac{ \mathrm{tr}\left( K^{p-1}\partial_X K\right) \cdot      \mathrm{tr}\left( K^{p-1}\partial_{Y} K\right)      }{   \left(\mathrm{tr}K^p\right)^{\frac{2p-1}{p}}}.\label{B4-1}
		\end{align}
		Using that $K$ is diagonal at ${\bf{z}}$, we simplify the summation in \eqref{B4-1} as follows:
		\begin{align}
			\sum_{j=0}^{p-2}\mathrm{tr}\left( K^j \partial_{Y}K K^{p-2-j}\partial_X K  \right)= \sum_{j=0}^{p-2} \sum_{l,m} k_l^j \partial_{Y}K_l^m k_m^{p-2-j}\partial_X K_m^l. \label{B4-2}
		\end{align}
		Notice that on the right-hand side of the above equality $k_l^j k_m^{p-2-j}$ is a geometric series. So the summation can be easily computed. Denote 
		\begin{equation}
			\sum_{j=0}^{p-2} k_l^jk_m^{p-2-j}=S_{l,m}.     
		\end{equation}
		Then we get
		\begin{equation}
			S_{l,m}= 
			\begin{cases}
				\frac{k_l^{p-1}-k_m^{p-1}}{k_l-k_m}        &  \text{if } k_l \neq k_m ,\\
				(p-1)\cdot k_l^{p-2}         & \text{if }  k_l=k_m.
			\end{cases}
		\end{equation}
		Therefore, we obtain
		\begin{equation}
			\frac{\sum_{j=0}^{p-2}k_l^j k_m^{p-2-j}}{\left(\mathrm{tr}K^p\right)^{\frac{p-1}{p}}}=
			\begin{cases}
				\frac{p-1}{k_1 W_p} & \text{if }  k_l=k_m=k_1,\\
				\frac{(p-1)k_m^{p-2}}{k_1^{p-1}W_p} &\text{if } k_l=k_m\neq k_1,\\
				\frac{k_l^{p-1}-k_m^{p-1}}{(k_l-k_m)k_1^{p-1}W_p} & \text{if }  k_l\neq k_m.
			\end{cases}
		\end{equation}
		In the above, we denote 
		\begin{equation}
			W_p= \left[  1+\left( \frac{k_2}{k_1}\right)^p +\cdots +\left( \frac{k_n}{k_1}\right)^p \right]^{\frac{p-1}{p}}.
		\end{equation}
		The assumption that $k_1 > k_j$, for $j\ge 2$, implies that as $p\rightarrow \infty$, $W_p \rightarrow 1$.
		When $k_l=k_m\neq k_1$, as $p\rightarrow \infty$, we get
		\begin{equation}
			\frac{(p-1) k_m^{p-2}}{k_1^{p-1}W_p} \rightarrow 0.
		\end{equation}
		When $k_l\neq k_1$, $k_m \neq k_1$, as $p\rightarrow \infty$, we get
		\begin{equation}
			\frac{k_l^{p-1}-k_m^{p-1}}{(k_l-k_m) k_1^{p-1}W_p}=\frac{\left( \frac{k_l}{k_1}\right)^{p-1}-\left( \frac{k_m}{k_1}\right)^{p-1}}{(k_l-k_m)k_1W_p} \rightarrow 0.
		\end{equation}
		When $k_l=k_1$ and $k_l\neq k_m$, as $p\rightarrow \infty$, we get
		\begin{equation}
			\frac{k_l^{p-1}-k_m^{p-1}}{(k_l-k_m)k_1^{p-1}W_p}= \frac{1-\left( \frac{k_m}{k_1}\right)^{p-1}}{(k_1-k_m)W_p} \rightarrow \frac{1}{k_1-k_m}.
		\end{equation}
		When $k_m=k_1$ and $k_l\neq k_m$, as $p\rightarrow \infty$, we get
		\begin{equation}
			\frac{k_l^{p-1}-k_m^{p-1}}{(k_l-k_m)k_1^{p-1}W_p}= \frac{\left( \frac{k_l}{k_1}\right)^{p-1}-1}{(k_l-k_1)W_p} \rightarrow \frac{1}{k_1-k_l}.
		\end{equation}
		
		In the summation (\ref{B4-2}), the only possible term that goes to infinity as $p\rightarrow \infty$ is the one corresponding to $m=l=1$, but this term cancels with the last term of (\ref{B4-1}). Then we get
		\begin{align}
			&\frac{\sum_{j=0}^{p-2}k_1^j\partial_{Y}K_1^1 k_1^{p-2-j}\partial_X K_1^1 }{ \left(\mathrm{tr}K^p\right)^{\frac{p-1}{p}}} -(p-1) \frac{ \mathrm{tr}\left( K^{p-1}\partial_X K\right) \cdot      \mathrm{tr}\left( K^{p-1}\partial_{Y} K\right)      }{   \left(\mathrm{tr}K^p\right)^{\frac{2p-1}{p}}}.\\
			&=(p-1)\left[    \frac{\partial_X K_1^1 \partial_{Y}K_1^1 }{k_1W_p}-  \frac{ \mathrm{tr}\left( K^{p-1}\partial_X K\right) \cdot      \mathrm{tr}\left( K^{p-1}\partial_{Y} K\right)      }{   \left(\mathrm{tr}K^p\right)^{\frac{2p-1}{p}}}  \right]       .
		\end{align}
		We denote the last term in the above equality by $(p-1)(E_p-F_p)$. Employing the notation $W_p$, we get
		\begin{align*}
			F_p = \frac{1}{k_1 (W_p)^{\frac{2p-1}{p}}} \left[\partial_X K_1^1+ \sum_{j>1}\partial_X K_j^j \left( \frac{k_j}{k_1} \right)^{p-1} \right]\left[ \partial_{Y} K_1^1+\sum_{j>1} \partial_{Y} K_j^j\left( \frac{k_j}{k_1} \right)^{p-1} \right].
		\end{align*}
		So $E_p$ and $F_p$ both converge to $\frac{\partial_X K_1^1 \partial_{Y}K_1^1 }{k_1}$ exponentially as $p\rightarrow \infty$. As a result, $(p-1)(E_p-F_p) \rightarrow 0$. Consequently, we get 
		\begin{equation}
			\lim_{p\rightarrow \infty} \partial_{XY}H^{(p)}=\partial_{XY}K^1_1 +\sum_{l\neq1} \frac{\partial_X K_1^l \partial_{Y} K_{l}^1}{k_1-k_{l}}+\sum_{l\neq 1}\frac{\partial_{Y} K_1^l\partial_{X} K_{l}^1}{k_1-k_{l}}.
		\end{equation}
		So, (\ref{B3-2}) is proved.
		
		To prove  (\ref{B3-3}), we notice that the convergence in (\ref{B3-2}) is exponential, i.e., at ${\bf{z}}$,
		\begin{equation}
			\left|\partial_{XY}H^{(p)}-\left(\partial_{XY}K^1_1 +\sum_{\theta\neq1} \frac{\partial_X K_1^\theta \partial_{Y} K_{\theta}^1}{k_1-k_{\theta}}+\sum_{\theta\neq 1}\frac{\partial_{Y} K_1^\theta \partial_{X} K_{\theta}^1}{k_1-k_{\theta}}\right) \right| \le C_1 \e^{-C_2p}.
		\end{equation}
		With $C_1$, $C_2$ depending on $\max \left( \left| K_i^j \right| + \left| \nabla K_i^j\right| \right)$ and $(k_1-k_{\theta})$.
		So the convergence in (\ref{B3-2}) is uniform in a small neighborhood of ${\bf{z}}$. This then implies the right-hand side of (\ref{B3-2}) is indeed $\partial_{XY} \lambda_{max}(K)$.
		
		Now we prove the equations regarding minimal eigenvalues. Let $S=K^{-1}$. Since $K$ is diagonal at the point ${\bf{z}}$,
		\begin{equation}
			S=\mathrm{diag} \left( \frac{1}{k_1},\frac{1}{k_2}, \cdots, \frac{1}{k_n} \right) \overset{\triangle}{=} \mathrm{diag} \left( s_1,s_2, \cdots, s_n \right).
		\end{equation}
		We assume $k_1< k_2 \le \cdots \le k_n$, so $s_1>s_2\ge \cdots \ge s_n$.
		Denote the maximal eigenvalue of $S$ by $g$ and the minimal eigenvalue of $K$ by $f$, then $f=\frac{1}{g}$,
		\begin{align}
			\p f&=-\frac{1}{g^2} \p g, \label{B9-1}\\
			\pq f&= -\frac{1}{g^2}\pq g+2\frac{1}{g^3}\p g \q g. \label{B9-2}
		\end{align}
		In (\ref{B9-1}), at the point ${\bf{z}}$, computing $\p g$ using (\ref{B1-1}) yields 
		\begin{align}
			\p g =\p S_1^1 = - \left( K^{-1} \p K K^{-1 }\right)^1_1 =-\frac{\p K_1^1 }{(k_1)^2}.
		\end{align}
		Then, at the point ${\bf{z}}$, we get
		\begin{equation}
			\p f =-\frac{1}{g^2}\p g =  \p K^1_1.
		\end{equation}
		It follows that 
		\begin{align}
			2\frac{1}{g^3}\p g \q g =2 (k_1)^3 \frac{1}{(k_1)^2}\p K^1_1 \cdot \frac{1}{(k_1)^2}\q K_1^1 = \frac{2}{k_1}\p K^1_1 \q K_1^1 .\label{B10-1}
		\end{align}
		Then we use (\ref{B1-2}) to compute 
		\begin{align}
			\pq g = \left( \pq S \right)^1_1 + \sum_{j>1}\frac{\left( \p S \right)^j_1 \left( \q S \right)^1_j}{s_1-s_j}+\sum_{j>1}\frac{\left( \q S \right)^j_1 \left( \p S \right)^1_j}{s_1-s_j}. \label{B10-2}
		\end{align}
		Replacing $S$ by $K^{-1}$, we get $\p S=-K^{-1}\p K K^{-1}$. Hence, 
		\begin{align}
			\left( \p S \right)^j_1 &=- \frac{ \left( \p K \right)^j_1 }{k_1k_j} ,\\
			\pq S&=K^{-1}\left(-\pq K +\p K K^{-1}\q K +\q K K^{-1}\p K \right)K^{-1}.
		\end{align}
		Since $K$ is diagonal, we obtain
		\begin{align}
			&-\frac{1}{g^2}\left( \pq S \right)_1^1 \\
			&=  \pq K^1_1-\sum_{j>1} \frac{(\p K_1^j) (\q K_j^1)}{k_j}  
			-\sum_{j>1} \frac{(\q K_1^j) (\p K_j^1)}{k_j} -\frac{2(\q K_1^1) (\p K_1^1)}{k_1}.
		\end{align}
		The last term in the above equality cancels with (\ref{B10-1}). Recalling (\ref{B9-2}),  we get 
		\begin{align*}
			&\pq f= -\frac{1}{g^2}\pq g+2\frac{1}{g^3}\p g \q g\\
			&    =\left(\pq K \right)_1^1 -\sum_{j>1}\left[(\p K_1^j) (\q K_j^1) +(\q K_1^j) (\p K_j^1)  \right] \left[ \frac{1}{k_{j}} +\frac{     \frac{1}{{(k_1)^2(k_j)^2}}       }{\frac{1}{k_1}-\frac{1}{k_j}}\cdot k_1^2 \right].
		\end{align*}
		The coefficient of $\left[(\p K_1^j) (\q K_j^1) +(\q K_1^j) (\p K_j^1)  \right]$ can be simplified as follows
		\begin{align}
			\frac{1}{k_{j}} +\frac{     \frac{1}{{(k_1)^2(k_j)^2}}       }{\frac{1}{k_1}-\frac{1}{k_j}}\cdot k_1^2= \frac{1}{k_j-k_1}.
		\end{align}
		So, (\ref{B2-2}) is proved.
	\end{proof}
	
	\section{Power convexity of solutions to complex Monge-Amp\`{e}re equation in complex dimension one}
	In this appendix, we prove our main theorem in the case of  complex dimension one. Although the argument is essentially one-dimensional, it captures the main ideas of the general case. In particular, the construction and computations presented here closely parallel those in higher dimensions, while remaining simpler and more transparent.
	\begin{theorem}
		Let $\Omega$ be a strictly convex bounded domain in $\mathbb{C}$ with smooth boundary and $u$ be the classical solution of Problem \ref{prob:Dirichlet}.
		Then $-\sqrt{-u}$ is strictly convex in $\Omega$.
	\end{theorem}
	
	\begin{proof} 
			We only need to prove an apriori estimate that if $-\sqrt{-u}$ is convex then it's strictly convex. Then the result follows from a deformation argument.
		
		For dimension one, the complex Monge-Amp\`{e}re equation is 
		\begin{equation}\label{C1-1}
			u_{z\bar{z}} =1 .
		\end{equation}
		By the maximum principle and the boundary condition $u|_{\partial \Omega=0}$, we get $-u>0$ in $\Omega$.  Let $v:=-\sqrt{-u}$. 
		Then Hessian matrices of $v$ are the functions
		\begin{align}
			v_{z\bar{z}} =\frac{1}{2\sqrt{-u}}\left[ u_{z\bar{z}}+ \frac{u_z u_{\bar{z}}}{-2u} \right],
			\quad \text{  and  } \quad  v_{zz}=\frac{1}{2\sqrt{-u}}\left[ u_{zz}+\frac{u_z u_{z}}{-2u} \right] .
		\end{align}
		Let $h(u)= \frac{1}{-2u}$. We set the functions 
		\begin{align}\label{C1-2}
			\alpha = u_{z\bar{z}}+hu_zu_{\bar{z}}, \text{ and } \beta =u_{zz}+hu_zu_{z}.
		\end{align}
		Let 
		\begin{equation}
			\sigma= 1-\frac{|\beta|^2}{\alpha^2}.
		\end{equation}
		By Lemma A.6 of \cite{hu25} or Lemma \ref{lem:determinant_lemma} of this paper, the function $v$ is convex if and only if $\alpha>0$ and $ \sigma>0$.  Since the function $\alpha$ is positive in $\Omega$, we only need to prove the function $\sigma>0$ in $\Omega$. 
		We want to show that 
		\begin{equation}
			\zz \sigma \le C \cdot(\sigma+|\nabla \sigma|),
		\end{equation}
		which is, using the notation of this paper,
		\begin{equation}
			\zz \sigma \lesssim 0.
		\end{equation}
		Since $\sigma$ is invariant under complex linear change of coordinates, we may arrange that, at the point of computation,
		\begin{equation}
			\alpha=1, \quad \beta \in \mathbb{R},\quad  0<\beta <1
		\end{equation}
		In what follows, all computations are carried out at this fixed point. In particular, at the point of computation we have
		\begin{equation}
			\alpha=1, \quad \beta \sim 1.
		\end{equation}
		Before computing $\zz \sigma$, we derive some differential relations for $\alpha$ and $\beta$.
		Since $u_{z\bar{z}}=1$, we get $u_{z\bar{z}z}=0$. Using (\ref{C1-2}) and $h'=2h^2$, we apply $\partial_z$ to the function $\alpha$ and get
		\begin{align}
			\pz\alpha&= u_{z\bar{z}z}+2h^2u_z^2u_{\bar{z}}+hu_{zz}u_{\bar{z}}+hu_z u_{z\bar{z}} \\
			&=0+h(hu_{z}^2+u_{zz})u_{\bar{z}}+h(hu_{z}u_{\bar{z}}+u_{z\bar{z}})u_z\\
			&=h\beta u_{\bar{z}}+h\alpha u_z \label{a(1-1)}\\
			&\sim hu_{\bar{z}}+hu_z.\label{C(1.1)}
		\end{align}
		Similarly, we have 
		\begin{align}
			\pbz\beta =2h\alpha u_{z} \sim 2hu_z.\label{C(2.1)}
		\end{align}
		Then we apply $\partial_{\bar{z}}$ to the function $\pz\alpha$ in \eqref{a(1-1)} and get
		\begin{align}
			\zz \alpha &=h' u_{\bar{z}}\beta u_{\bar{z}}+h\pbz\beta u_{\bar{z}}+h\beta u_{\bar{z}\bar{z}}+h'u_{\bar{z}}\alpha u_z+h\pbz\alpha u_z+h\alpha u_{z\bar{z}} \\
			&=2h^2\beta u_{\bar{z}}^2+2h^2\alpha |u_z|^2+h\beta (\overline{\beta}-hu_{\bar{z}}^2) \\
			&\hspace{2ex}+2h^2\alpha|u_z|^2+h^2\overline{\beta} u_z^2+h^2\alpha |u_z|^2+h\alpha(\alpha-h|u_z|^2) \\
			&=h(\alpha^2+|\beta|^2)+4h^2\alpha|u_z|^2+h^2\beta u_{\bar{z}}^2+h^2\overline{\beta}u_z^2. \\
			&\sim 2h+4h^2|u_z|^2+h^2 u_{\bar{z}}^2+h^2u_z^2.
		\end{align}
		Similarly, we have 
		\begin{align}
			\zz\beta&=2h\alpha \beta+4h^2\alpha u_z^2+2h^2\beta|u_z|^2 \sim 2h+4h^2 u_z^2+2h^2|u_z|^2.
		\end{align}
		Using the first derivatives of $\alpha$ and $\beta$, we get the following relations 
		\begin{align}
			\zz \alpha-   \pz\alpha \pbz\alpha&\sim  2h+2h^2|u_z|^2, \label{C2.1.1}\\
			\zz\alpha-\pbz\beta \pz\overline{\beta}&\sim  2h+h^2(u_z^2+u_{\bar{z}})^2 ,\label{C2.1.2}\\       
			\zz \beta-\pz\alpha \pbz\beta&\sim 2h+2h^2 u_z^2.\label{C2.1.3}
		\end{align}
		For convenience, we set 
		\begin{equation}
			\mathscr{B}= \pz\beta-\frac{\beta \pz\alpha}{\alpha}.
		\end{equation}
		Then we have 
		\begin{align}
			\frac{\mathscr{B}}{\alpha}=\pz\left(\frac{\beta}{\alpha}\right).
		\end{align}
		Let $\mu= \alpha -\frac{|\beta|^2}{\alpha}$, then we get 
		\begin{equation}\label{a(2-2)}
			\pz \mu = \pz \alpha - \frac{\mathscr{B}\overline{\beta}}{\alpha}-\frac{\beta \pz\overline{\beta}}{\alpha}.
		\end{equation}
		Recalling the definition of $\sigma$, we get 
		\begin{equation}\label{a(2-1)}
			\pz \sigma= \pz \left( \frac{\mu}{\alpha}\right)=\frac{\pz \mu}{\alpha}-\frac{\mu \pz \alpha}{\alpha^2} \sim 0. 
		\end{equation}
		Combining \eqref{a(2-1)} and $\mu\sim 0$ yields $\pz \mu \sim 0$. Then from \eqref{a(2-2)},  we find 
		\begin{equation}
			\pz \alpha  - \pz\overline{\beta} \sim \mathscr{B}.
		\end{equation}
		Subtracting the complex conjugate of  \eqref{C(2.1)} from \eqref{C(1.1)}, we get 
		\begin{equation}\label{C3.1}
			h(u_{\bar{z}}-u_z) \sim \mathscr{B}.
		\end{equation}
		We apply $\partial_{\bar{z}}$ to the function $ \mu_{z}$ in \eqref{a(2-2)} and get
		\begin{align}
			\zz\mu &=\zz \alpha - \pbz \mathscr{B} \frac{\overline{\beta}}{\alpha}-\mathscr{B} \pbz \left( \frac{\overline{\beta}}{\alpha} \right) -\frac{\pbz\beta \pz \overline{\beta}}{\alpha} - \beta  \pbz \left( \frac{\pz \overline{\beta} }{\alpha} \right) \\
			&\sim ( \zz \alpha-   \pz\alpha \pbz\alpha)+(\zz\alpha-\pbz\beta \pz\overline{\beta}) -(\zz \beta-\pz\alpha \pbz\beta)  \\
			&\hspace{3ex}-\overline{(\zz \beta-\pz\alpha \pbz\beta)} -|\mathscr{B}|^2.
		\end{align}
		We plug in \eqref{C2.1.1}, \eqref{C2.1.2} and \eqref{C2.1.3}, and get 
		\begin{equation}
			\zz \mu \sim h^2 |u_z-u_{\bar{z}}|^2 -|\mathscr{B}|^2.
		\end{equation}
		Then using \eqref{C3.1}, we have 
		\begin{equation}
			\zz \mu \sim 0.
		\end{equation}
		Combining $\mu\sim 0$, $\pz \mu \sim 0$ and $\zz \mu \sim 0$ yields
		\begin{equation*}
			\zz \sigma =\zz \left( \frac{\mu}{\alpha} \right)= \zz \mu \frac{1}{\alpha}+ \pz \mu \pbz \left( \frac{1}{\alpha}\right)+\pbz \mu \pz \left( \frac{1}{\alpha}\right) +\mu \zz \left( \frac{1}{\alpha}\right)\sim 0.
		\end{equation*}
	\end{proof}


\begin{thebibliography}{99}
	
	

	
	
	\bibitem[BG09] {Bian-Guan2009}B. Bian, P. Guan, {\em A microscopic convexity principle for nonlinear partial	differential equations}, Invent. math. 177, 307–335 (2009).
		\bibitem[BGMX11] {Bian-Guan-Ma-Xu2011}B. Bian, P. Guan, X. Ma, L. Xu, {\em A constant rank theorem for quasiconcave solutions of fully
		nonlinear partial differential equations}, Indiana University Mathematics Journal, Vol. 60, No. 1 (2011), pp. 101-119.
	\bibitem[CGM07]{Caffarelli-Guan-Ma2007} L. Caffarelli, P. Guan, X. Ma, {\em A constant rank theorem for solutions of fully nonlinear
		elliptic equations}, Communications on Pure and Applied Mathematics 60(12):1769 - 1791.
	

	
	\bibitem[CF85]{Caffarelli-Friedman1985} L. Caffarelli, A. Friedman, {\em Convexity of Solutions of Semilinear Elliptic Equations, }Duke Math. J. 52(2): 431-456 (June 1985). 
	
	\bibitem[CJX25]{Chen-Jia-Xiong2025}
	C, Chen, H, Jia, J, Xiong,
	{\em Strict Power Convexity of Solutions to Fully Nonlinear Elliptic Partial Differential Equations in Two Dimensional Convex Domains}.
	{J. Differential Equations} (2025) 419:505--517.
	
	\bibitem[CLM26]{ComplexSigmaEigen} C. Chen, J. Li, X. Ma, {\em Brunn--Minkowski Inequality for the First Complex -Hessian Eigenvalue,} 	arXiv:2606.25678.
	
	\bibitem[GX13]{Guan-Xu2013} P.Guan, L. Xu, {\em Convexity estimates for level sets of quasiconcave solutions
		to fully nonlinear elliptic equations}, J. Reine Angew. Math.vol. 2013, no. 680, 2013, pp. 41-67. 
	
\bibitem[HJ13]{Horn-Johnson2013} R. Horn, C. Johnson, Matrix Analysis, Cambridge Univ. Press, Cambridge, Second Edition 2013.

\bibitem[H21]{hu2021} J. Hu, {\em An Obstacle for Higher Regularity of Geodesics in the Space of K\"ahler Potentials.} International Mathematics Research Notices, Volume 2021, Issue 15, August 2021, Pages 11493–11513.
\bibitem[H24A]{Hu-metric-lower-bound24} J. Hu, {\em A Metric Lower Bound Estimate for  Geodesics in the Space of K\"ahler Potentials.} The Journal of Geometric Analysis (2024) 34:225.
\bibitem[H24B]{hu-maxrank-2024}J. Hu, {\em A Maximum Rank Theorem for Solutions to the Homogenous Complex Monge–Amp\`ere Equation in a $\EC$-Convex Ring.} Calc.Var.(2024) 63:166.
\bibitem[H25]{hu25}J. Hu, {\em The Preservation of Convexity by Geodesics in the Space of K\"ahler Potentials on Complex Affine Manifolds}. 	Math. Annalen (2025) 393:1635–1681.
\bibitem[HS26]{HuSheng} J. Hu, L. Sheng, {\em Convexity of the Potential Function of the Einstein-Kähler Metric on a Convex Domain}, 	arXiv:2603.10530.

\bibitem[LMS26]{RealSigma}J. Li, X. Ma, P. Salani, {\em A Brunn--Minkowski inequality for the Hessian eigenvalue in convex domain}, 	arXiv:2606.22847.

\bibitem[MX08]{maxu}  X. N. Ma and L. Xu. {\em The convexity of solution of a class Hessian equation in bounded convex
	domain in R3}. J. Funct. Anal., 255(7):1713–1723, 2008.

\bibitem[SW16]{SzekelyhidiWeinkove2016} G. Sz\'ekelyhidi, B. Weinkove, {\em On a constant rank theorem for nonlinear elliptic {PDE}s}, Discrete Contin. Dyn. Syst. 2016, Volume 36, Issue 11: 6523-6532. 

\bibitem[ZZ26]{ZhangZhou} W. Zhang, Q. Zhou, {\em Power convexity of solutions to complex Monge-Amp\`ere equation in $\EC^2$.}   arXiv:2505.11002v3.
	\end{thebibliography}
\end{document}